\documentclass[a4paper,11pt]{amsart}
\usepackage{amsfonts}
\usepackage{amsmath,amsthm,cases,mathtools,mathrsfs} 
\usepackage{latexsym}
\usepackage{array}
\usepackage{amssymb}
\usepackage{enumerate}

\usepackage{color}
\usepackage[latin1]{inputenc}

\usepackage{hyperref}

\theoremstyle{plain} 
\newtheorem{theorem}{Theorem}[section]
\newtheorem{lemma}[theorem]{Lemma}
\newtheorem{proposition}[theorem]{Proposition}
\newtheorem{corollary}[theorem]{Corollary}

\theoremstyle{remark}
\newtheorem{remark}[theorem]{Remark}
\newtheorem{example}[theorem]{Example}

\newtheorem*{lem*}{Lemma}
\newtheorem*{sublem*}{Sublemma}
\newtheorem*{remark*}{Remark}
\newtheorem*{NB*}{NB}

\newcommand{\Lbb}{  \mathbb{L}   }
\newcommand{\Mbb}{  \mathbb{M}   }
\newcommand{\R}{  \mathbb{R}   }
\newcommand{\C}{  \mathbb{C}   }
\newcommand{\Z}{  \mathbb{Z}   }
\newcommand{\N}{  \mathbb{N}   }

\newcommand{\T}{  \mathbb{T}   }

\def\ir{{\rm i}}

\newcommand{\Cc}{  \mathcal{C}   }
\newcommand{\cC}{  \mathcal{C}   }
\newcommand{\D}{  \mathcal{D}   }

\newcommand{\F}{  \mathcal{F}   }

\newcommand{\J}{  \mathcal{J}   }
\newcommand{\U}{  \mathcal{U}   }
\renewcommand{\L}{  \mathcal{L}   }
\newcommand{\M}{  \mathcal{M}   }

\renewcommand{\O}{  \mathcal{O}   }
\renewcommand{\P}{  \mathcal{P}   }

\newcommand{\Tc}{  \mathcal{T}   }

\newcommand{\Zc}{  \mathcal{Z}   }
\newcommand{\Sc}{  \mathcal{S}   }

\newcommand{\om}{  \omega   }
\newcommand{\Om}{  \Omega   }
\newcommand{\ga}{\gamma   }
\newcommand{\s}{  \sigma   }

\newcommand{\ka}{  \kappa   }
\renewcommand{\r}{  \rho   }

\renewcommand{\phi}{  \varphi  }
\newcommand{\eps}{\varepsilon}

\newcommand{\de}{  \delta   }
\renewcommand{\b}{  \beta   }

\renewcommand{\a}{  \alpha   }

\newcommand{\la}{  \lambda_a   }

\newcommand{\La}{  \Lambda  }

\newcommand{\diag}{\operatorname{diag}}

\newcommand{\meas}{\operatorname{meas}}

\newcommand{\be}{\begin{equation}}
\newcommand{\ee}{\end{equation}}
\newcommand{\ben}{\begin{equation*}}
\newcommand{\een}{\end{equation*}}
\newcommand{\ban}{\begin{align*}}
\newcommand{\ean}{\end{align*}}
\numberwithin{equation}{section}

\newcommand{\dd}{  \text{d}   }

\newcommand{\nazz}{ (\nabla_\rho\cdot {\mathfrak z} ) }

\newcommand{\subsubsubsection}[1]{\paragraph{\it{#1}}}
 \author{ Beno\^it Gr\'ebert}
\address{Laboratoire de Math\'ematiques Jean Leray, Universit\'e de Nantes, UMR CNRS 6629\\
2, rue de la Houssini\`ere \\
44322 Nantes Cedex 03, France}
\email{benoit.grebert@univ-nantes.fr}

\author{ Victor Vila\c{c}a Da Rocha }
\address{BCAM - Basque Center of Applied Mathematics\\
Alameda de Mazarredo 14 \\
48009 Bilbao, Bizkaia (Basque Country), Spain}
\email{vvilaca@bcamath.org}

\title[Small amplitude  solutions for coupled NLS systems]
{Stable and unstable  time quasi periodic solutions for a system of coupled NLS equations}

\begin{document}

\begin{abstract}
We prove that  a system of coupled nonlinear Schr\"odinger  equations on the torus exhibits both stable and unstable small KAM tori.
In particular the unstable tori are related to a beating phenomena which has been proved recently in \cite{GPT}. This is the first example of unstable tori for a 1d PDE.

\medskip

 
 \end{abstract}

\keywords{ Small amplitude solutions, modulational instabilities, Hamiltonian systems, nonlinear PDE, KAM theory.}
\thanks{The authors are partially supported   by the  grant  BeKAM ANR -15-CE40-0001-02 and by the Centre Henri Lebesgue ANR-11-LABX-0020-01. 
The second author is also supported by the grant ERCEA Advanced Grant 2014 669689 - HADE
}

\maketitle

\tableofcontents

\section{Introduction}

We consider the system of coupled nonlinear Schr\"odinger  equations  on the torus
\begin{equation}\begin{dcases}
i\partial_tu+\partial_{xx}u=\left|v\right|^2u,& \quad (t,x)\in \mathbb{R}\times\mathbb{T}, \\ 
i\partial_tv+\partial_{xx}v=\left|u\right|^2v.
\end{dcases}\label{NLSC}\end{equation}
This system is Hamiltonian when considered on the phase space $(u,\bar u,v,\bar v)\in (L^2(\T))^4$ endowed with the symplectic form $-idu\wedge d\bar u-idv\wedge d\bar v$. The Hamiltonian of the system is given by 
\begin{equation*}
 H:=\displaystyle\int_{\mathbb{T}}\left(|u_x|^2+|v_x|^2\right)dx + \displaystyle\int_{\mathbb{T}}|u|^2|v|^2dx.
\end{equation*}
In section \ref{secappli} we will consider a slightly more general case where
we add a higher order perturbation $R_5$ (see \eqref{sysgen}). In the introduction we prefer to focus on the simplest case. We also remark that all our results concern small amplitude solutions and thus the sign in front of the linearity doesn't affect our results (but we need the same sign in both line of \eqref{NLSC} to conserve the Hamiltonian structure).


In order to take profit of the geometry of the torus, we write the Fourier series expansion of $u$, $\overline{u}$, $v$ and $\overline{v}$:
\begin{align*}
u(t,x)&=\sum\limits_{j\in\mathbb{Z}}a_j(t)e^{ijx}, \quad \quad \quad \overline{u}(t,x)=\sum\limits_{j\in\mathbb{Z}}\bar a_j(t)e^{-ijx},\\
v(t,x)&=\sum\limits_{j\in\mathbb{Z}}b_j(t)e^{ijx}, \quad \quad \quad \overline{v}(t,x)=\sum\limits_{j\in\mathbb{Z}}\bar b_j(t)e^{-ijx}.\\
\end{align*}
In this variables, the symplectic structure becomes 
\begin{equation*}
 -i\sum_j da_j\wedge d\bar a_j - i\sum_j db_j\wedge d\bar b_j,
\end{equation*}
and the Hamiltonian $H$ of the system reads
\begin{equation}\label{defNP}
 H(a,\bar a,b,\bar b)=\sum\limits_{j\in\mathbb{Z}}j^2(a_j\bar a_j+b_j\bar b_j)+\sum\limits_{\substack{i,j,k,l\in\mathbb{Z} \\ i+j=k+l}}a_kb_l\overline{a}_i\overline{b}_j=P_2+P_4.
\end{equation}
In this article we are interested in the persistence of two dimensional linear invariant tori: given $p, q\in\Z$ and $a_p,b_q\in\C$, 
\begin{equation}\label{sol}\begin{dcases}
 u(x,t)&=a_{p}e^{ipx}e^{-ip^2t},\\ v(x,t)&=b_{q}e^{iqx}e^{-iq^2t}, 
 \end{dcases}\end{equation}
is a  solution to the linear system associated to the quadratic Hamiltonian $P_2$. 
Equivalently we can say that for any $c\in\R^2$ and any $p,q\in\Z$, the torus $\Tc_{c}(p,q)=\{|a_p|^2=c_1,\ |b_q|^2=c_2  \}$ is  invariant under the flow of $P_2$. 
We prove (see Theorem \ref{KAMunstable}) that for $p\neq q$, for $\rho$ in a Cantor set of full measure in $[1,2]^2$ and for $\nu$ small enough, 
the non linear Hamiltonian $P_2+P_4$ admits an invariant torus close to $\Tc_{\nu\rho}(p,q)$. 
Furthermore we prove that these tori are {\it linearly unstable}: the system \eqref{NLSC} linearized around $\Tc_{\nu\rho}(p,q)$ admits one hyperbolic direction. 
In other words, for $|a_p|^2=\nu\rho_1,\ \ |b_q|^2=\nu\rho_2 $ with $\rho$ in a Cantor set  and $\nu$ small enough, \eqref{NLSC} admits an unstable small amplitude quasi periodic solution close to \eqref{sol}. Precisely we prove in section \ref{subsecunstable}:
 \begin{theorem}\label{coro1}Fix $p\neq q$ and $s>1/2$. There exists $\nu_0>0$ and for $0<\nu<\nu_0$ there exists $\Cc_\nu\subset [1,2]^2$ asymptotically of full measure (i.e. $\lim_{\nu\to 0}\meas([1,2]^2\setminus \Cc_\nu )=0$) such that for $\r\in\Cc_\nu$ there exists a quasi periodic solution $(u,v)$ of \eqref{sysgen} of the form
\begin{equation*}\begin{dcases}
 u(x,t)&=\sum_{j\in\mathbb{Z}}u_j(t\om)e^{ijx},\\
 v(x,t)&=\sum_{j\in\mathbb{Z}}v_j(t\om)e^{ijx},
 \end{dcases}\end{equation*}
where $U(\cdot)=(u_j(\cdot))_{j\in\Z}$ and $V(\cdot)=(v_j(\cdot))_{j\in\Z}$ are analytic  functions from $\T^2$ into $\ell^2_s$ satisfying uniformly in $\theta\in\T^2$
\begin{equation*}
\begin{dcases}
\big||u_p(\theta)|-\sqrt{\nu\rho_1}\big|^2+\sum_{j\neq p}(1+j^2)^s|u_j(\theta)|^2=\mathcal{O}(\nu^3),\\
\big||v_q(\theta)|-\sqrt{\nu\rho_2}\big|^2+\sum_{j\neq q}(1+j^2)^s|v_j(\theta)|^2=\mathcal{O}(\nu^3) 
\end{dcases}
\end{equation*}
and where $\om\equiv\om(\r)\in\R^2$ is a nonresonant frequency vector that satisfies
\begin{equation*}
 \om=(p^2,q^2)+\mathcal{O}(\nu^\frac32).
\end{equation*}
Furthermore this solution is linearly unstable.
\end{theorem}

It is not the first time that one exhibits unstable KAM tori (see for instance \cite{EGK,PP}) but it turns out that it is the first example in a one dimensional context. 
This unstable behavior is to be compared with to the {\it modulational instabilities} extensively studied by physicists since fifty years  (see \cite{BF, ZO} 
and \cite{Hung17} for a coupled case different from ours). \\
We also prove (see Theorem \ref{KAMstable} and Corollary \ref{coro2}) the persistence of the invariant torus $\Tc_{\rho}(p,p)$ but, in this case, the torus in linearly stable. \\
We stress out that, although the existence of invariant tori requires a lot of assumptions (and in particular we have to assume that $\rho$ is in a Cantor set), 
when an invariant torus exists, its stability or instability  is only related to the choice of the modes.

\medskip

The result is obtained by putting $H$ in a normal form $h+f$ suitable to  apply a singular KAM theorem (see section \ref{KAM}) which is essentially contains in \cite{EGK}. 
Notice that $P_2$ is totally resonant and thus is not adapted to a KAM procedure. The general idea, coming from \cite{KP}, consists in using $P_4$ to break the resonances. 
First we apply a Birkhoff procedure (see Section \ref{birk}) to eliminate the non resonant part of $P_4$
$$(P_2+P_4)\circ\tau= P_2+Z_4+\text{higher order term}.$$
Then in sections \ref{subsecunstable} and  \ref{subsecstable} we calculate the effective part of $Z_4$ in two different cases. This step is highly related to the choice of the torus that we want to perturb. \\ 
We notice that we could consider more general tori of any finite dimension (i.e. quasi periodic solutions constructed on finitely many linear modes).
The instability of the corresponding torus will appear when one excites initially two modes $a_p$ and $b_q$ with $p\neq q$. 
To simplify the presentation we prefer to focus on two dimensional tori.\\
The strategy and the proofs are inspired by \cite{EGK}. The aim of this paper is to present these recent technics in a simpler case leading to a surprising result: instability seems typical even in 1d context.

\medskip

We end this introduction with a remark linking instability of KAM tori and existence of a beating effect.
Taking advantage of the resonances between the linear frequencies and of the coupling by the quartic term $P_4$,  Gr\'ebert-Paturel-Thomann proved in \cite{GPT} (see also \cite{V}) that \eqref{NLSC} exhibits a beating phenomena: roughly speaking when you consider initial data of the form
\begin{equation}\label{descr} 
 \begin{dcases}
 u(0,x)&=a_{p}(0)e^{ipx}+a_{q}(0)e^{iqx},\\
  v(0,x)&=b_{p}(0)e^{ipx}+b_{q}(0)e^{iqx}\,,
  \end{dcases}
  \end{equation}
  with $p\neq q$ and $|a_{p}(0)|=|b_{q}(0)|=\ga \eps$, $|a_{q}(0)|=|b_{p}(0)|=(1-\ga)\eps$ for $0<\ga<1/2$ and $\eps$ small enough, then the four modes exchange energy periodically, i.e. they are close to
 \begin{equation*} 
\begin{dcases}
 |a_{q}(t)|^{2}&= |b_{p}(t)|^{2}=K_{\gamma}(\eps^{2} t),\\
 |a_{p}(t)|^{2}&= |b_{q}(t)|^{2}=1-K_{\gamma}(\eps^{2} t),
 \end{dcases}
 \end{equation*} 
 where $K_\ga$ a $2T-$periodic function ($T\sim|\ln \ga|$)  which satisfies $K_\ga(0)=\ga$ and $K_\ga(T)=1-\ga$. 
In \cite{GPT} the result is proved only  for a finite but very long time but in view of \cite{HP}, we can expect that such beating solution exists for all time. \\
In this work we consider the case $\ga=0$ which corresponds to a two dimensional invariant torus, $\Tc_{(\sqrt{|a_{p}(0)|},\sqrt{|b_{q}(0)|})}$ for the linear system and 
we prove that the KAM theory applies, i.e. that the non linear Hamiltonian $P_2+P_4$ admits invariant tori close to $\Tc_{(\sqrt{|a_{p}(0)|},\sqrt{|b_{q}(0)|})}$. 
Nevertheless, as we have seen, the tori are linearly unstable: when  linearized  around the torus, the system presents two hyperbolic directions which corresponds to the two other modes of the beating picture above.  This means that the beating effect is related to the instability of the two tori: the one construct on the modes $a_p,b_q$ and the one constructed on the modes $a_q,b_p$. Actually the monomial in $P_4$ which makes possible the beating effect, namely $a_p\bar b_p \bar a_q b_q$, is also responsible for the instability of the tori.\\
The beating phenomena has also be exhibited  for the quintic NLS (see \cite{GT}) and for cubic NLS with some special nonlinearities (see \cite{GV}). 
It turns out that following the same line we could prove the existence of unstable KAM Tori in these two other cases.
The main problem in both cases will be to verify that the hypotheses of the KAM theorem are satisfied, which will lead to computations similar but different from those of Appendix 
\ref{AppA}.

\section{An abstract KAM theorem}
In this section we state a KAM theorem adapted to our problem. 
We consider a Hamiltonian $H=h_0+ f$, where  $h_0$ is a quadratic Hamiltonian in normal form
\begin{equation}\label{NN}
h_0 =\Om(\r)\cdot r+\sum_{\a\in\Zc}\Lambda_\a(\r)|\zeta_\a|^2 .\end{equation}
Here 
\begin{itemize}
\item $\rho$ is a parameter in $\D$, which is a compact in the space $\R^{n}$;
\item $r\in\R^n$ are the actions corresponding to the internal modes
 $(r,\theta)\in (\R^n\times\T^n,dr\wedge d\theta)$;
\item $\L$ and $\F$ are respectively infinite and finite sets, 
 $\Zc$ is the disjoint union $\L\cup\F$;
\item  $\zeta=(\zeta_s)_{s\in\Zc}\in \C^\Zc$ are the external modes endowed with the standard complex symplectic structure $-\ir \mathrm d\zeta\wedge \mathrm d\bar\zeta$ .  
The external modes decomposes in an infinite part $\zeta_\L=
   (\zeta_s)_{s\in\L}$, corresponding  to elliptic directions, which means that $\Lambda_\a\in\R$ for $\a\in\L$, and  a finite part  $\zeta_\F=(\zeta_s)_{s\in\F}$ corresponding to hyperbolic directions, which means that $\Im\Lambda_s\neq 0$ for $s\in\F$;
\item the mappings 
 \be\label{properties}\left\{\begin{array}{ll}
\Om:\D\to\R^{n},&\\
\La_\a:\D\to \C\,,&\quad  \a\in \Zc,
\end{array}\right.\ee
are smooth. 
\item   $f=f(r,\theta, \zeta;\rho)$ is the perturbation,
 small compare to the integrable part $h_0$. 
\end{itemize}

\subsection{Setting}
We define precisely the spaces and norms:

 \smallskip

\noindent {\bf Clustering structure on $\L$.}
We assume that $\L$ has a clustering structure: 
$$\L=\cup_{j\in \N}\L_j$$
where $\L_j$, $j\in\N$, are finite sets of cardinality $d_j\leq d<+\infty$. If $\a\in\L_j$ we denote $[\a]=\L_j$ and $w_\a=j$.
We consider $\F$ as an extra cluster of $\Zc=\L\cup\F$ and for $\a\in\F$ we set $w_\a=1$.

\begin{example}\label{ex21} In the second case of NLS systems (see Subsection \ref{subsecstable}),  we will set $\L=\Zc\subset\Z\times\{\pm\}$, $\F=\emptyset$ and 
$\zeta_{\a},\ \a\in\Zc$ will denote all the external modes: $\zeta_{j_+}=a_j, \ \zeta_{j_-}=b_j$ for $j\in\Z\setminus\{p\}$ and $a_p$, $b_p$ are the internal modes. 
The clusters are given by $[j\pm]=\{j_+,j_-,-j_+,-j_-\}$ and $d_j=4$ for $j\neq |p|$, and an extra cluster is given by $[-p]=\{-p_+,-p_-\}$ and $d_{-p}=2$. 
Then we set $w_{j\pm}=|j|$.
\end{example}

\begin{example}\label{ex22} In the first case of NLS systems (see Subsection \ref{subsecunstable}),  we will set $\Zc\subset\Z\times\{\pm\}$. 
If $a_p$, $b_q$ ($p\neq q$) are the two internal modes then $\L=\Z\setminus\{p,q\}\times\{\pm\}$ and  $\zeta_{\a},\ \a\in\L$ will denote all the elliptic external modes:  
$\zeta_{j_+}=a_j, \ \zeta_{j_-}=b_j$ for $j\in\Z\setminus\{p,q\}$. 
As we will see we have two hyperbolic external modes, $b_p$ and $a_q$ then $\F=\{p_-,q_+  \}$. 
The clusters of $\L$ are given by $[j\pm]=\{j_+,j_-,-j_+,-j_-\}$ with $d_j=4$ for $j\neq |p|,|q|$.
If $p\neq-q$, we have to add  two extra clusters given by $[-p]=\{-p_+,-p_-\}$ and  $[-q]=\{-q_+,-q_-\}$ with $d_{-p}=d_{-q}=2$. 
Then we set $w_{j\pm}=|j|$.
\end{example}

 \smallskip

\noindent {\bf Linear space.}
Let  $s\geq 0$, we consider the complex  weighted $\ell_2$-space
$$
Z_s=\{\zeta=(\zeta_\a\in\C,\ \a\in \Zc)\mid \|\zeta\|_s<\infty\},
$$
where
$$
\|\zeta\|_s^2=\sum_{\a\in\Zc}|\zeta_\a|^2 w_\a^{2s}.$$
We provide the spaces $Z_s\times Z_s$, $s\geq 0$, with the symplectic structure $-\ir \dd\zeta\wedge\dd \bar\zeta$. \\
Similarly we define 
 $$
Y_s=\{\zeta_\L=(\zeta_\a\in\C,\ \a\in \L)\mid \|\zeta\|_s<\infty\},
$$
endowed with the same norm and symplectic structure restricted to indexes in $\L$.

 \smallskip
 
  \noindent {\bf Infinite matrices.}
 For the elliptic variables, we denote by $\M_s$ the set of infinite matrices $A:\L\times \L\to \C$ such that 
 $A$ maps linearly $Y_s$ into $Y_s$. We provide  $\M_s$ with the operator norm
$$|A|_s=\|A\|_{\L(Y_s,Y_s)}.$$

 
   We say that a matrix $A\in\M_s$ is in normal form if it is block diagonal and Hermitian, i.e.
  \be\label{ANF}A_{\b}^\a=0 \quad \text{for}\quad [\a]\neq[\b] \quad \text{and}\quad A_{\b}^\a=\overline{A_{\a}^\b}\text{ for }\a,\b\in\L.\ee
In particular, we use that if $A\in\M_s$ is in normal form, its eigenvalues are real.

 \smallskip

\noindent {\bf A class of  Hamiltonian functions.}
Let us fix any $n\in\N$. 
On the space
$$
\C^n\times \C^n \times (Z_s\times Z_s)$$
we define the norm
$$\|(r,\theta,z)\|_s=\max(|r|, |\theta|, \|z\|_s).$$
For $\sigma>0$ we denote
$$
\T^n_\sigma=\{\theta\in\C^n: | \Im \theta|<\sigma\}/2\pi\Z^n.
$$
%
For $\sigma,\mu\in(0,1]$ and $s\ge0$ we set
$$
\O^s(\s,\mu)= \{r\in\C^n: |r|<\mu^2\}\times \T^n_\s\times \{z\in Z_s\times Z_s: \|z\|_s<\mu\}.
$$
We will denote points in $\O^s(\s,\mu)$ as $x=(r,\theta,z)$.
A
 function defined on a domain $\O^s(\s,\mu)$, is called {\it real} if it gives real values to real
arguments $x=(r,\theta,z)$ with $r,\theta$ reals and $z=(\zeta,\bar\zeta)$.\\
Let 
$$
\D=\{\rho\}\subset \R^p 
$$ 
be a compact set of positive Lebesgue  measure. This is the set
of parameters upon which will depend our objects. Differentiability of functions on $\D$
is understood in the sense of Whitney. So $f\in C^1(\D)$ if it may be extended to a $C^1$-smooth
function $\tilde f$ on $ \R^p$, and $|f|_{ C^1(\D)}$ is the infimum of  $|\tilde f|_{ C^1(\R^p)}$,
taken over all $C^1$-extensions $\tilde f$ of $f$. \\
Let 
 $f:\O^0(\s,\mu)\times\D\to \C$ be a $C^1$-function, 
 real holomorphic in the first variable $x$,
  such that for all $\rho\in\D$
 \begin{align}\label{regularize}
 \O^{s}(\s,\mu)\ni x\mapsto \nabla_z f(x,\rho)\in Z_{s}\times Z_s
 \end{align}
and
 $$
 \O^{s}(\s,\mu)\ni x\mapsto\nabla^2_{\zeta_\L \bar\zeta_\L} f(x,\rho)\in \M_{s}$$
are real holomorphic functions\footnote{In fact by Cauchy's theorem the analyticity of $\nabla_z f$ on  $\O^{s}(\s,\mu)$ with value in $Z_s\times Z_s$ yields the analyticity of 
$\nabla^2_{zz} f$ on $ \O^{s}(\s',\mu)$ with values in $\L(Z_s\times Z_s,Z_s\times Z_s)$, and thus the analyticity of $\nabla^2_{\zeta_\L \bar\zeta_\L} f$ on $ \O^{s}(\s',\mu)$ 
with values in $\M_s$, for any $\s'<\s$  . We conserve the two hypotheses to mimic \cite{EGK} and \cite{GP} where an additional property on the Hessian of $f$ was required. }. 
We denote by $\Tc^{s}(\s,\mu,\D)$ this set of functions.
 For a function $f\in \Tc^{s}(\s,\mu,\D)$ we define
 the norm 
 $$[f]^{s}_{\s,\mu,\D}$$
 through 
$$\sup
\max( |\partial^j_\r f(x,\r)|,\mu \|\partial^j_\r \nabla_z f(x,\r)\|_{s},
\mu^2|\partial^j_\r \nabla^2_{\zeta_\L \bar\zeta_\L} f(x,\r)|_{s}),
$$
where the supremum is taken over all
$$
j=0,1,\ x\in \O^{s}(\s,\mu),\ \rho\in\D.$$
When the function $f$ does not depend on $(r,\theta)$ neither on $\r$ we denote $f\in \Tc^{s}(\mu)$.
\begin{example}\label{exf}
Let $\Zc=\Z\setminus\{k_1,\cdots,k_n \}$ and $g$ an analytic function from a neighborhood of the origin in $\C^2$ into $\C$. We define a Hamiltonian $f$ by
$$
\O^{s}(\s,\mu)\ni x\mapsto f(x):=\int_{\T}g(\hat u(t),\hat{\bar u}(t))dt$$
with
\begin{align*}
\hat u(t)&= \sum_{\ell=1,\cdots,n}r_{\ell}e^{i\theta_\ell}e^{ik_\ell t}+\sum_{\a\in\Zc}\zeta_\a e^{i\a t},\\
\hat{\bar u}(t)&= \sum_{\ell=1,\cdots,n}r_{\ell}e^{-i\theta_\ell}e^{-ik_\ell t}+\sum_{\a\in\Zc}\bar\zeta_\a e^{-i\a t}.
\end{align*}
We verify  that for  $s>1/2$ and $\s>0$, $\mu>0$ small enough $f\in \Tc^{s}(\s,\mu,\D)$ (here $f$ does not depend on $\r$). A precise proof is given in the Appendix A of \cite{EGK}. We can recall here the basis of the proof: we have 
$$\frac{\partial f}{\partial \zeta_\a}=\int_{\T}\partial_1 g(\hat u(t),\hat{\bar u}(t))e^{i\a t}dt$$
and since $\hat u$ and $\hat{\bar u}$ have their Fourier coefficients in $\ell^2_s$, they are both functions in the Sobolev space $H^s$ which in turns implies that $t\mapsto \partial_1 g(\hat u(t),\hat{\bar u}(t))$ is an $H^s$ function and thus its Fourier coefficients are in $\ell^2_s$.
\end{example}

\smallskip

\noindent {\bf Jet-functions.}
For any function $f\in \Tc^{s}(\s,\mu,\D)$ we define its jet $f^T(x)$, $x=(r,\theta,z)$, as the following 
 Taylor polynomial of $f$ at $r=0$ and $z=0$
\be
\label{jet}
f(0,\theta,0)+d_r f(0,\theta,0)[r] +d_z f(0,\theta,0)[z]+\frac 1 2 d^2_zf(0,\theta,0)[z,z].
\ee
Functions of the form  $f^T$ will be called {\it jet-functions.}

\smallskip

\noindent {\bf A restricted class of Hamiltonian functions.} We will need to avoid certain monomials in the jet of our perturbation (see for instance the proof of Proposition \ref{Melni}). 
For that purpose, we will say that $f\in \Tc_{\rm res}^{s}(\s,\mu,\D)$ if there exits a constant $M$ such that for all $k\neq 0$ and all $\a,\b\in\L$ with $[\a]=[\b]$
\be \label{mom}e^{ik\cdot\theta} \zeta_\a\bar \zeta_\b\in f^T \implies \a=\b \text{ or } |w_\a|\leq M|k|. \ee
We will see that this condition is always satisfied for  systems of nonlinear Schr\"odinger equations of type \eqref{NLSC}, or more generally \eqref{sysgen} with the assumption \eqref{g}, 
thanks to the conservation of the momentum (see Lemma \ref{MoM}). We also remark that such restriction was not needed in \cite{EGK} or \cite{GP} since in these papers the perturbation is regularizing.

\smallskip

\noindent {\bf Poisson brackets.} The Poisson brackets of two Hamiltonian functions 
is defined by
$$
\{f,g\}= \nabla_\theta f\cdot\nabla_r g -\nabla_rf\cdot\nabla_\theta g-i\langle \nabla_z f,J\nabla_z g\rangle\,.
$$
\begin{lemma}\label{lemma-poisson} Let $s> 1/2$.
Let $f\in\Tc^{s}(\s,\mu,\D)$ and $g\in \Tc^{s}(\s,\mu,\D)$ be two jet functions then for any $0<\s'<\s$   we have $\{f,g\}\in \Tc^{s}(\s',\mu,\D)$ and
$$[\{f,g\}]_{\s',\mu,\D}^{s}\leq C(\s-\s')^{-1}\mu^{-2}[f]_{\s,\mu,\D}^{s}[g]_{\s,\mu,\D}^{s}.$$
Furthermore if $f,g\in \Tc_{\rm res}^{s}(\s,\mu,\D)$ then $\{f,g\}\in \Tc_{\rm res}^{s}(\s',\mu,\D)$.
\end{lemma}
The proof follows as in \cite{GP} Lemma 4.3. This stability result is fundamental to apply the KAM scheme.

\smallskip

\noindent {\bf Normal form.}
A quadratic Hamiltonian function is on normal form if it reads
\begin{equation}\label{HamilNF}
 h=V(\r)\cdot r+\langle \zeta_\L,A(\r)\bar\zeta_\L\rangle+\frac12\langle z_\F,K(\r)z_\F\rangle
\end{equation}
for some vector function $V(\r)\in\R^n$, some matrix functions  
 $A(\r)\in\M_s$ on normal form ( see \eqref{ANF}) and $K(\r)$ is a matrix $\F\times\F\to gl(2,\C)$ symmetric  in the following sense\footnote{This symmetry comes from the matrix representation that we chose for Hessian functions and the Schwarz rule.}: $K_\b^\a={}^tK_\a^\b$.

\subsection{Hypothesis 
}\label{subsechypo}

The following three hypotheses concerned only the quadratic Hamiltonian $h_0$. The first one is related to the asymptotic of $\Lambda_\a$, the two other are non resonances conditions.
 
\noindent{\bf Hypothesis A0} (spectral asymptotic.) There exists $C>0$ such that
$$| \Lambda_\a-|w_\a|^{2}|\leq C, \ \forall\, \a\in \L\,.$$

\medskip

\noindent{\bf Hypothesis A1} (Conditions on external frequencies.) \\
 There exists $\delta>0$ such that  for all $\r\in\D$ we have
 
(a)\; The elliptic frequencies don't vanish:
$$|\Lambda_\a|\ge \delta\;\ \forall\, \a\in \L \,;$$
and the hyperbolic frequencies have a non vanishing imaginary part
$$|\Im\Lambda_\a|\ge \delta\;\ \forall\, \a\in \F \,;$$

(b) \; The difference between two external frequencies doesn't vanish except if they are in the same cluster:
\begin{align*}|\Lambda_\a(\rho) -\Lambda_\b(\rho)| &\ge \delta\quad \forall \a,\b\in \Zc \text{ with }[\a]\neq[\b] \,;
\end{align*}

(c) The sum of two elliptic frequencies doesn't vanish:
 $$|\Lambda_\a(\rho) +\Lambda_\b(\rho)| \ge \delta \text{ for all } \a,\b\in \L.$$

\noindent{\bf Hypothesis A2} (Transversality conditions.)\\
These conditions express that the small divisors cannot stay in a resonant position:\\
There exists $\delta>0$ such that  for all $\tilde\Om(\cdot)$ $\delta-$close in  $C^1$ norm from $\Om(\cdot)$ and for all $k\in\Z^n\setminus\{0\}$:\\
(i) either 
$$|\tilde\Om(\r)\cdot k|\geq \delta\quad  \forall  \r\in\D,$$
or there exits a unit vector\footnote{The notation $\nazz f(\r)$ means  that we take the gradient of $f$ at the point $\r$ in the direction $\mathfrak z$. } ${\mathfrak z}={\mathfrak z}(k)\in\R^n$  such that 
$$\nazz\big(\tilde\Om(\r)\cdot k\big)\geq \delta\quad  \forall  \r\in\D.
$$
(ii) for all $\a\in\L$ either
$$|\tilde\Om(\r)\cdot k+\Lambda_\a(\r)|\geq \delta\quad  \forall  \r\in\D,$$
or  there exits a unit vector ${\mathfrak z}={\mathfrak z}(k)\in\R^n$  such that 
$$\nazz\big(  \tilde\Om(\r)\cdot k+\Lambda_\a(\r) \big) \geq \delta\quad  \forall \r\in\D,$$
(iii)  for all $\a,\b \in\L$ either
$$ 
|\tilde\Om(\r)\cdot k+\Lambda_a(\r)\pm \Lambda_b(\r) |\geq \delta\quad  \forall  \r\in\D, a\in[\a],b\in[\b],$$  
or there exits a unit vector ${\mathfrak z}={\mathfrak z}(k)\in\R^n$  such that 
$$\nazz\big(\tilde\Om(\r)\cdot k+\Lambda_a(\r)\pm \Lambda_b(\r) \big)\geq \delta\quad  \forall  \r\in\D,a\in[\a],b\in[\b]\,.
$$
(iv)  for all $\a,\b \in\F$ 
$$ 
|\tilde\Om(\r)\cdot k+\Lambda_\a(\r)\pm \Lambda_\b(\r) |\geq \delta\quad  \forall  \r\in\D,$$  
\begin{remark}\label{F} 
Hypothesis A2 (iv) may appear not reasonable since we require it for all $k\in\Z^n\setminus\{0\}$ with an upper bound that does not depend on $k$. Typically it can be satisfied if $
|\Im (\Lambda_\a(\r)\pm \Lambda_\b(\r)) |\geq \delta$ or using a momentum argument as in remark \ref{momentum} below.\\
This hypothesis is used to simplify the treatment of the hyperbolic directions. For a more general hypothesis (requiring higher regularity with respect to the parameter) see  \cite{EGK}.  \end{remark}

 Hypotheses A1 and A2 are used (see Proposition \ref{Melni}) to control small denominators of the form 
 $$ \begin{array}{llll}
&\Om\cdot k&\quad  \forall  k\neq 0,\\
&\Om\cdot k+\Lambda_\a&\quad  \forall \a\in \Zc, \ k\in\Z^n,\\
&\Om\cdot k+\Lambda_\a+ \Lambda_\b &\quad  \forall \a,\b\in \Zc, \ k\neq 0,\\
&\Lambda_\a+ \Lambda_\b &\quad  \forall \a,\b\in \L,\\
&\Om\cdot k+\Lambda_\a- \Lambda_\b &\quad  \forall \a,\b\in \Zc, \ k\neq 0, \\
&\Lambda_\a- \Lambda_\b &\quad  \forall \a,\b\in \Zc \text{ with } [\a]\neq[\b]
\end{array}$$  
which in turns are used to kill the following monomials of the jet of the perturbation  $f$
$$ \begin{array}{llll}
&e^{ik\cdot\theta}  &\forall  k\neq 0,\\
& e^{ik\cdot\theta} \zeta_\a ,\ e^{ik\cdot\theta} \bar \zeta_\a &\forall \a\in \Zc, \ k\in\Z^n,\\
&e^{ik\cdot\theta} \zeta_\a \zeta_\b,\ e^{ik\cdot\theta} \bar \zeta_\a\bar \zeta_\b \quad & \forall \a,\b\in \Zc, \ k\neq 0,\\
& \zeta_\a \zeta_\b,\  \bar \zeta_\a\bar \zeta_\b\quad & \forall \a,\b\in \L,\\
& e^{ik\cdot\theta} \zeta_\a\bar \zeta_\b \quad  &\forall \a,\b\in \Zc, \ k\neq 0, \\
& \zeta_\a\bar \zeta_\b\quad  &\forall \a,\b\in \Zc \text{ with } [\a]\neq[\b].
\end{array}$$  
\begin{remark}\label{momentum} 
If $f$ preserves some symmetries, and if these symmetries are preserved by the KAM procedure, then some monomials will never appear in the perturbation terms and the corresponding small 
denominator has not to be control. 
This can be used to relax Hypotheses A1 and A2.  For instance if $f$ commutes with the mass $M=\sum_{\a\in\Zc}|\zeta_\a|^2$ then monomials of the form $\zeta_\a\zeta_\b$ cannot appear in the jet of $f$ since $\{M,\zeta_\a\zeta_\b\}=-2i\zeta_\a\zeta_\b\neq0$. Thus for such $f$ Hypothesis A1 (c) has not to be satisfied. \\
In Appendix \ref{AppA}  this remark will be crucial. \end{remark}

\subsection{Statement and comments on the proof}
We recall that we consider a Hamiltonian $H=h_0+ f$, where $h_0$ is the quadratic Hamiltonian in normal form given by \eqref{NN}.
\begin{theorem}\label{KAM}
Assume that hypothesis A0, A1, A2 are satisfied\footnote{Hypotheses A1 and A2 can be partially relaxed according to Remark \ref{momentum}. }
and that $f\in\Tc_{res}^{s}(\s,\mu,\D)$ with $s>1/2$. Let $\ga>0$, there exists a constant $\eps_0>0$ such that if
\be\label{eps} [f]^{s}_{\s,\mu,\D}\leq \eps_0\delta\quad \text{and \quad}\eps:=[f^T]^{s}_{\s,\mu,\D}\leq \eps_0 \delta^{1+\ga},\ee
then there exists a Cantor set $\D'\subset\D$ asymptotically of full measure (i.e. $\meas \D\setminus\D'\to 0$ when $\eps\to 0$) and for all $\r\in\D'$ there exists a symplectic change of variables $\Phi: \O^{s}(\s/2,\mu/2)\to\O^{s}(\s,\mu)$ such that for $\r\in\D'$
$$(h_0+ f)\circ \Phi=  h+  g$$
with $ h=\langle\om(\r), r\rangle +\langle \zeta_\L,  A(\r)\bar\zeta_\L\rangle +\frac 1 2 \langle z_\F,  K(\r)z_\F\rangle $  on normal form (see~\eqref{HamilNF}) and 
$g\in\Tc^{s}(\s/2,\mu/2,\D')$ with $ g^T\equiv 0$. Furthermore there exists $C>0$ such that for all $\r\in\D'$
$$|\om-\Omega|\leq C\eps,\, |A-\diag(\Lambda_\a,\ \a\in\L)|\leq C\eps \; \text{and}\; |JK-\diag(\Lambda_\a,\ \a\in\F)|\leq C\eps.$$
As a dynamical consequences $\Phi(\{0\}\times\T^n\times\{0\})$ is an invariant torus for $h_0+f$ and this torus  is linearly stable if and only if $\F=\emptyset$.
\end{theorem}

Theorem \ref{KAM} is a  normal form result, we can explain its dynamical consequences.  First, for $\r\in\D'$, the torus
$\{0\}\times\T^n\times\{0\}$ is invariant by the flow of $h+g$ and thus the torus
 $\Phi(\{0\}\times\T^n\times\{0\})$ is invariant by the flow of $h_0+f$ and the dynamics on it is the same as that of $ h$.\\
Moreover, the linearized equation on this torus reads 
$$
\left\{\begin{array}{l}
\dot \zeta_\L=-iA\zeta_\L-i\partial^2_{r\bar\zeta}g(0,\theta_0+\om t,0)\cdot r,\\
\dot z_\F=-iJKz_\F-iJ\partial^2_{rz}g(0,\theta_0+\om t,0)\cdot r,\\
\dot\theta=\partial^2_{rz}g(0,\theta_0+\om t,0)\cdot z+\partial^2_{rr}g(0,\theta_0+\om t,0)\cdot r,\\
\dot r=0.
\end{array}\right.
$$
Since $A$ is on normal form the eigenvalues of the $\zeta_\L$-linear part in the first line are purely imaginary (see \eqref{ANF}). 
Since furthermore $JK$ is sufficiently close to the diagonal matrix $\diag(\Lambda_\a,\ \a\in\F)$, the eigenvalues of the  $\zeta_\F$-linear part in the second line have a non 
vanishing real part. 
Finally, the last term in the two first lines is a bounded term, independent on $\zeta$ (and on $z=(\zeta,\bar\zeta)$), and thus doesn't play a role in the linear stability.
Therefore the invariant torus is linearly stable if and only if $\F=\emptyset$.

\medskip

The proof is standard and we don't include it in this article. Nevertheless it is not a direct consequence of an existing KAM theorem.
Essentially Theorem \ref{KAM} is a mix between the KAM theorem proved in \cite{KP} (see also~\cite{Posch}) and the one proved in \cite{EGK} 
(see also \cite{GP} for a proof in Sobolev regularity or \cite{Mou} for a 1d version). 
In Theorem~\ref{KAM} and in the KAM theorem proved in \cite{KP} or \cite{Posch}, we have the same asymptotics of the frequencies (Hypothesis A0) which simplifies the proof and doesn't require regularizing perturbation as in \cite{EGK}.  Nevertheless in \cite{KP} or \cite{Posch} the frequencies are non resonant and thus there is no clustering. So we need \cite{EGK} and the clustering structure to prove Theorem \ref{KAM}.

Let us explain why Hypothesis A0, A1, A2 allow to control the so called small divisors.
The KAM proof  is based on an iterative procedure that requires  
to solve a  homological equation at each step. Roughly speaking, it consists in inverting an infinite dimensional matrix whose eigenvalues are the so-called small divisors:
\begin{align*}
& \om\cdot k\quad  k\in\Z^n\setminus\{0\},\\
&\om\cdot k +\lambda_\a\quad  k\in\Z^n,\ \a\in\Zc,\\
&\om\cdot k +\lambda_\a\pm\lambda_\b\quad  k\in\Z^n,\ \a,\b \in\Zc
\end{align*}
where $\om=\om(\r)$ and $\la=\la(\r)$ are small perturbations (changing at each KAM step) of the original frequencies $\Om(\r)$ and $\Lambda_a(\r), a\in\L$.

 The transversality condition (Hypothesis A2) ensures that for most values of $\r$,  all these eigenvalues are far away from zero (at least at the first step):
 \begin{proposition}\label{Melni} Let  $M,N\geq1$ and $0<\ka\leq \delta$. Assume Hypothesis A0, A1, A2.
 Then there exists a closed subset $\D'\equiv\D'(\ka,N)\subset \D$ satisfying 
\begin{align*}\meas\  \D\setminus \D'&\leq C\delta^{-1} \ka M^2N^{n+2},\end{align*}
  such that for all $\r\in\D'$, for all $|k|\leq N$ and for all $\a,\b \in \Zc$
\begin{align}
\label{D00} &|\Om(\r) \cdot k|\geq  \ka, \text{ except if  } k=0,\\
\label{D11}&|\Om(\r) \cdot k+\Lambda_\a(\r)|\geq \ka, \\
&\label{D22}|\Om(\r)\cdot k+\Lambda_\a(\r)+ \Lambda_\b(\r) |\geq \ka,
\end{align}
and for all $\r\in\D'$, for all $|k|\leq N$ and for all  $\a,\b\in\L$ such that either $[\a]\neq[\b]$ or $[\a]=[\b]$ and $w_\a\leq M|k|$
\begin{align}
&\label{D33}|\Om(\r)\cdot k+\Lambda_\a(\r)- \Lambda_\b(\r) |\geq \ka.
\end{align}
\end{proposition}
Notice that in \eqref{D33}, the case $[\a]=[\b]$ and $w_\a\geq M|k|$, has not to be controlled if $f\in\Tc_{res}$ (see \eqref{mom} and remark \ref{momentum}).\\
Let us recall the following classical result
\begin{lemma}\label{estim}(see for instance \cite{EK} appendice A) Let $I$ be an open interval and let $f:I \to\R$
be a $\cC^{1}$-function satisfying
$$|{f'}(x)|\ge \delta,\quad \forall x\in I.$$
Then,
$$
\meas \{x\in I: |{f(x)}|<\eps\}\le C \frac\eps\de.$$
\end{lemma}

\proof[Proof of Proposition \ref{Melni}]
Let us prove the estimates \eqref{D22} and \eqref{D33}, the other two being similar but easier. We notice that  \eqref{D22} and \eqref{D33} hold true for $\a,\b\in\F$ by hypothesis A2 (iv) and for $\a\in\L$ and $b\in\F$ by hypothesis A1(a). So it remains to consider the case $\a,\b\in\L$.\\
Let us begin with \eqref{D22}.  Let us fix $k\neq0,\ \a,\b\in\Zc$, by Hypothesis A2 we have, either 
$$|\Om(\r)\cdot k+\Lambda_\a(\r)+ \Lambda_\b(\r)|\geq\ka \quad \forall\r\in\D,$$
or
$$\nazz\big(\Om(\r)\cdot k+\Lambda_\a(\r)+ \Lambda_\b(\r) )\geq \delta\quad   \forall\r\in\D\,.$$
  Then  we have using Lemma \ref{estim}
\begin{equation*}
\meas \{\r\mid |\Om(\r)\cdot k+\Lambda_\a(\r)+ \Lambda_\b(\r) |< \ka\}\leq C \ka\de^{-1}
\end{equation*}
where $C$ does not depend on $k,\a,\b$.
On the other hand, in view of Hypothesis A0, we remark that $\Om(\r)\cdot k+\Lambda_\a(\r)+ \Lambda_\b(\r) $ can be small only if $w_\a,w_\b\leq C|k|^{1/2}$. 
Therefore 
\begin{align*}
\meas \{\r\mid |\Om(\r)\cdot k+\Lambda_\a(\r)+ \Lambda_\b(\r) |< \ka, \mbox{ for some } &|k|\leq N \mbox{ and } \a,\b\in\Zc\} \\
&\leq C \frac\ka\de N^{n}N^{\frac{1}{2}}N^{\frac{1}{2}}.
\end{align*}
The proof of \eqref{D33} is similar except that, since $|j^2-\ell^2|\geq 2|j|-1$ for any integers $j\neq\ell$, we deduce that, when $[\a]\neq[\b]$,  $\Om(\r)\cdot k+\Lambda_\a(\r)- \Lambda_\b(\r) $ can be small only if $w_\a,w_\b\leq C|k|$. 
\endproof 
\begin{remark}
In the multidimensional case, the transversality condition is not enough to ensure \eqref{D33} and we have to add the so called second Melnikov condition in the list of hypothesis (see for instance \cite{EGK}). The problem comes from the fact that  in $\Z^d$ with $d\geq2$, $||j|^2-|\ell|^2|$ can be small even for large $j$ and $\ell$.
\end{remark}

\section{Birkhoff normal form}\label{birk}
In this section we apply a Birkhoff procedure to the Hamiltonian $P_2+P_4$ (see \eqref{defNP}) in order to eliminate the non resonant monomials.

\subsection{The framework of the study}
For $s\geq 0$ we define
$$\ell^2_{s}=\{x\equiv(x_k)_{k\in\Z}\in\C^\Z\mid ||x||_{s}<+\infty\},$$
where
$$||x||^2_{s}=\sum_{k\in\Z} (1+|k|^2)^s|x_k|^2.$$
We denote $x\star y$  the convolution of sequences in $\ell^2$: $(x\star y)_\ell=\sum_{i+j=\ell}x_iy_j$, and recall that for $s>1/2$, $\ell^2_{s}$ is a Hilbert algebra with respect to the convolution product and
\begin{equation}\label{conv}||x\star y||_s\leq c_s ||x||_s ||y||_s.\end{equation}
For a proof of this classic property, see \cite{KP}, Appendix A for $a=0$.\\
We consider the phase space $\P_s=\ell^2_s\times \ell^2_s\times\ell^2_s\times \ell^2_s\ni (a,\bar a,b,\bar b)$ endowed with the canonical symplectic structure 
\begin{equation*}
 -i\sum \mathrm{d}a_k\wedge \mathrm{d}\bar a_k-i\sum \mathrm{d}b_k\wedge \mathrm{d}\bar b_k.
\end{equation*}
Finally, we introduce $B_{s}(r)$, the ball of radius $r$ centered at the origin in~$\P_s$.
An interesting feature of the space $\P_s$ is the behavior of the Hamiltonian vectorfields of homogeneous bounded polynomials of $\P_s$, as we can see through the following lemma:
\begin{lemma}\label{XPanalytic}
Let   $s>1/2$. Let $P$ be a homogeneous polynomial of order $4$ on $\P_s$ of the form
\be\label{P}P(a,b,\bar a,\bar b)=\sum_{\substack{i,j,k,l\in\mathbb{Z} \\ i+j=k+l}}c_{i,j,k,\ell}\ a_ib_j\bar a_k\bar b_\ell\ee
with $|c_{i,j,k,\ell}|\leq M$ for all $(i,j,k,\ell)\in\J:=\{i,j,k,l\in\mathbb{Z}\mid  i+j=k+l \}$. Then the Hamiltonian vectorfield $X_P$ is analytic from $B_s(1)$, the unit ball in $\P_s$, into $\P_s$, with
$$||X_P(a,b,\bar a,\bar b)||_{s}\leq 4M||(a,b,\bar a,\bar b)||_{s}^3.$$
In particular $P\in\Tc^{s}(1)$. Furthermore $P$ commutes with $\Lbb$ and $\Mbb$.
\end{lemma}
Of course this lemma extends to polynomials of any order with bounded coefficients and zero momentum (the generalization to any order of the condition $i+j=k+\ell$). 
In particular this lemma shows that $X_{P_4}$, $X_{Z_4}$ and $X_{\chi_4}$ introduced in lemma below are vectorfields on $\P_s$. 
Notice that the linear vectorfield $X_{P_2}$ is unbounded on $\P_s$, since it takes values in $\P_{s-2}$.
\proof
As the coefficients of $P$ are bounded by $M$ we get
$$\left|\frac{\partial P}{\partial \bar b_\ell}\right|\leq M \sum_{(i,j,k,\ell)\in\J}|a_ib_j\bar a_k|.$$
Therefore writing
$$\sum_{(i,j,k,\ell)\in\J}|a_ib_j\bar a_k|= (a\star b\star\tilde{\bar a})_\ell,$$
where $\tilde x$ denotes the sequence defined by $\tilde x_j=x_{-j}$, we deduce using \eqref{conv}
$$\Big|\Big|\frac{\partial P}{\partial \bar b_\ell}\Big|\Big|_{s}\leq M ||(a,b,\bar a,\bar b)||_{s}^3.$$
We control the other partial derivatives of $P$ in the same way and we conclude that
$$||X_{P}(a,b,\bar a,\bar b)||_{s}\leq 4M ||(a,b,\bar a,\bar b)||_{s}^3.$$
\endproof

\subsection{The Birkhoff normal form result}
We recall that 
$$H=P_2+P_4$$
is defined by \eqref{defNP}. We also introduce the 
mass and momentum Hamiltonians:
\begin{equation*}
 \mathbb{L}=\sum\limits_{j\in\mathbb{Z}}(|a_j|^2+|b_j|^2), \quad \text{and} \quad 
 \mathbb{M}=\sum\limits_{j\in\mathbb{Z}}j(|a_j|^2+|b_j|^2)\,.
\end{equation*}
We notice that $H$ commutes with both $\Lbb$ and $\Mbb$:
$$\{H,\Lbb\}=\{H,\Mbb\}=0$$
which means that the Hamiltonian flow generated by $H$ preserves the mass and the momentum.

\begin{proposition}\label{BNF}
For $\varepsilon$ small enough, there exists a symplectic change of variables $\tau$ from the ball $B_{s}(\varepsilon)$ of $\P_s$ into $\U_\eps$ a neighborhood of the origin in $\P_s$ 
included in $B_s(2\varepsilon)$ such that:
\begin{equation}H^{B}:=H\circ\tau=P_2+Z_4+R_6,\label{Hbirk}\end{equation}
where : 
\begin{enumerate}[(i)]
	\item $P_2$ depends only on the actions: 
	\begin{equation*}
	P_2(a,\overline{a},b,\overline b)=\sum\limits_{j\in\mathbb{Z}}j^2(|a_j|^2+|b_j|^2).
	\end{equation*}
	\item $Z_4$ is the $4^{th}$ order homogeneous polynomial: 
	\begin{equation*}
	Z_4(a,\overline{a},b,\overline b)=\sum\limits_{\substack{k+l=i+j \\ k^2+l^2=i^2+j^2}}a_kb_l\overline{a_i}\overline{b_j}
	=\sum\limits_{\left\{k,l\right\}=\left\{i,j\right\}}a_kb_l\overline{a_i}\overline{b_j}.
	\end{equation*}
	In particular, $Z_4$ is resonant in the following sense:  $\left\{P_2,Z_4\right\}=0$.
	\item  $R_6$ is a Hamiltonian function in $\Tc^{s}(1)$ which satisfies 
	\begin{equation*}
	\left\|X_{R_6}(a,\overline{a},b,\overline b)\right\|_{s}\leq C\left\|(a,\overline{a},b,\overline b)\right\|^5_{s}\quad \text{for all}\quad 
	(a,\overline{a},b,\overline b)\in B_{s}(\varepsilon).	 
	\end{equation*}
	Furthermore $R_6$  commutes with $\Lbb$ and $\Mbb$.
	\item $\tau$ preserves the class $\Tc^{s}(\mu)$ for any $s>1/2$ and $\mu>0$ and it also conserves the mass and the momentum. Furthermore $\tau$  is close to the identity: there exists $C_{s}$ such that
	\begin{equation*}
	\left\|\tau(a,\overline{a},b,\overline b)-(a,\overline{a},b,\overline b)\right\|_{s}\leq C_{s}\left\|(a,\overline{a},b,\overline b)\right\|^3_{s} \quad 
	\text{for all} \quad (a,\overline{a},b,\overline b)\in B_{s}(\varepsilon).
	 \end{equation*}
\end{enumerate}
\end{proposition}

\begin{proof}
The idea is to search $\tau$ as the time 1 flow $\varphi_1$ of $\chi_4$ where $\chi_4$ is a Hamiltonian polynomial of order 4.
We write
\begin{equation*}
 \chi_4(a,\overline{a},b,\overline b):=\sum\limits_{p,q,r,s\in\mathbb{Z}}m(p,q,r,s)a_{p}\overline{a}_{q}b_{r}\overline{b}_{s}.
\end{equation*}
For $F$ a Hamiltonian, the Taylor expansion of $F\circ\varphi_t$ between the times $t=0$ and $t=1$ gives:
\begin{equation*}
 F\circ\tau:=F\circ\varphi_1=F+\left\{F,\chi_4\right\}+\displaystyle\int_0^1(1-t)\left\{\left\{F,\chi_4\right\},\chi_4\right\}\circ\varphi_t\mathrm{d}t.
\end{equation*} 
Applying this formula to $F=H=P_2+P_4$ we obtain
\begin{equation}\label{defZ4R6}
H\circ\tau=\underbrace{P_2}_{P_2}+\underbrace{P_4+\left\{P_2,\chi_4\right\}}_{Z_4}+
\underbrace{\left\{P_4,\chi_4\right\}+\displaystyle\int_0^1(1-t)\left\{\left\{H,\chi_4\right\},\chi_4\right\}\circ\varphi_t\mathrm{d}t}_{R_6}.
\end{equation}
Thus defined, the first point of the proposition is already checked.
We have now to choose the polynomial $\chi_4$ such as $\tau$, $Z_4$ and $R_6$ satisfy the hypothesis of Proposition \ref{BNF}.
According to the equation \eqref{defZ4R6}, we want to solve the so called homological equation 
\begin{equation}\label{eqhomol}
 \left\{\chi_4,P_2\right\}=P_4-Z_4,
\end{equation}
where $P_2$, $P_4$ are given by equation \eqref{defNP} and $Z_4$ is given by Proposition \ref{BNF}.
For the right hand-side polynomial term, we have
\begin{equation*}
 (P_4-Z_4)(a,\overline{a},b,\overline b)=\sum\limits_{\substack{p-q+r-s=0 \\ p^2-q^2+r^2-s^2 \neq0}}a_p\overline{a}_qb_r\overline{b}_s.
\end{equation*}
The left hand-side Poisson bracket term gives
\begin{align*}
\left\{\chi_4,P_2\right\}&=-i\displaystyle\sum\limits_{j\in\mathbb{Z}}\left(\dfrac{\partial \chi_4}{\partial a_j}\dfrac{\partial P_2}{\partial \overline{a}_j}
-\dfrac{\partial \chi_4}{\partial \overline{a}_j}\dfrac{\partial P_2}{\partial a_j}+\dfrac{\partial \chi_4}{\partial b_j}\dfrac{\partial P_2}{\partial \overline b_j}
-\dfrac{\partial \chi_4}{\partial \overline b_j}\dfrac{\partial P_2}{\partial b_j}\right)\\
&= -i\sum\limits_{p,q,r,s\in\mathbb{Z}}(p^2-q^2+r^2-s^2)m(p,q,r,s)a_p\overline{a}_qb_r\overline{b}_s.
\end{align*}
Therefore, in order to solve the homological equation \eqref{eqhomol}, it suffices to choose
\begin{equation*}
m(p,q,r,s)=
\begin{dcases}
 \dfrac{i}{(p^2-q^2+r^2-s^2)} \quad \text{if} 
\begin{dcases}
p-q+r-s=0, \\
p^2-q^2+r^2-s^2 \neq0,
\end{dcases}
\\
0  \quad \text{else.}
\end{dcases}
\end{equation*}
Thanks to the choice of $\chi_4$, we have constructed the polynomial $Z_4$. 
It remains to prove that $Z_4$ is resonant, i.e 
\begin{equation*}
 \left\{P_2,Z_4\right\}=0.
\end{equation*}
With the same computations as in the $\chi_4$ construction, we have:
\begin{equation*}
 \left\{P_2,Z_4\right\}=i\sum\limits_{p,q,r,s\in\mathbb{Z}}(p^2-q^2+r^2-s^2)z(p,q,r,s)a_p\overline{a}_qb_r\overline{b}_s,
\end{equation*}
where 
\begin{equation*}
z(p,q,r,s)=
\left\lbrace
\begin{array}{ccc}
1  & \mbox{if} & p-q+r-s=0 \quad \text{and} \quad p^2-q^2+r^2-s^2=0,\\
0 & \mbox{else.}&
\end{array}\right.
\end{equation*}
Therefore,
\begin{equation*}
 \forall (p,q,r,s)\in\mathbb{Z}^4, \; (p^2-q^2+r^2-s^2)z(p,q,r,s)=0,
\end{equation*}
and we have $\left\{P_2,Z_4\right\}=0.$ The proof of the second point of the proposition is completed.
By construction, the coefficients of $\chi_4$ are bounded. Thus, by Lemma \ref{XPanalytic}, we have
\begin{equation}\label{estXchi}
\left\|X_{\chi_4}(a,\overline{a},b,\overline b)\right\|_{s}=\mathcal{O}\left\|(a,\overline{a},b,\overline b)\right\|_{s}^3.
\end{equation}
We now want to prove that $\tau$ is well defined, i.e. we want to prove that the flow $\varphi_t$ is defined at least up to $t=1$.
For that purpose, the idea here is to use a bootstrap argument.
We introduce $T=T(a,\overline{a},b,\overline b)>0$ the existence time of the flow $\varphi_t(a,\overline{a},b,\overline b)$, and we consider a smaller time $0\leq s\leq T$.
Writing the fundamental theorem of calculus for $\varphi_t$, the flow of $\chi_4$, between the times $t=0$ and $t=s$, we obtain
\begin{align*}
 \varphi_s(a,\overline{a},b,\overline b)-\varphi_0(a,\overline{a},b,\overline b)=\displaystyle\int_0^s\dot{\varphi_t}(a,\overline{a},b,\overline b)\mathrm{d}t
 =\displaystyle\int_0^sX_{\chi}\left(\varphi_t(a,\overline{a},b,\overline b)\right)\mathrm{d}t.
\end{align*}
By definition, we have $\varphi_0=Id$. Thus, the equation \eqref{estXchi} implies (with $C>0$ a constant):
\begin{equation}\label{phisprocheidentite}
 \|\varphi_s(a,\overline{a},b,\overline b)-(a,\overline{a},b,\overline b)\|_s\leq C\int_0^s\left\|\varphi_t(a,\overline{a},b,\overline b)\right\|_{s}^3\mathrm{d}t.
\end{equation}
Let us choose $(a,\overline{a},b,\overline b)\in B_{\rho}(\varepsilon)$.
As long as 
\begin{equation*}
\|\varphi_s(a,\overline{a},b,\overline b)-(a,\overline{a},b,\overline b)\|_s\leq2\|(a,\overline{a},b,\overline b)\|_s, 
\end{equation*}
we have (using that $(a,\overline{a},b,\overline b)\in B_{\rho}(\varepsilon)$),
\begin{align*}
 \|\varphi_s(a,\overline{a},b,\overline b)-(a,\overline{a},b,\overline b)\|_s&\leq Cs(2\|(a,\overline{a},b,\overline b)\|_s)^3\\
 &\leq \left(8C\varepsilon^2s\right)\|(a,\overline{a},b,\overline b)\|_s
\end{align*}
For $\varepsilon$ small enough, we have $8C\varepsilon^2\leq1$. Thus we obtain
\begin{equation*}
  \|\varphi_s(a,\overline{a},b,\overline b)\|_s\leq (1+s)\|(a,\overline{a},b,\overline b)\|_s.
\end{equation*}
This bound is satisfied as soon as 
\begin{equation*}
 \|\varphi_s(a,\overline{a},b,\overline b)-(a,\overline{a},b,\overline b)\|_s\leq2\|(a,\overline{a},b,\overline b)\|_s \quad \text{and} \quad s\leq T(a,\overline{a},b,\overline b).
\end{equation*}
Therefore, by continuity, we have $T(a,\overline{a},b,\overline b)\geq1$ and
\begin{equation}\label{bootstrap}
\forall s\in[0,1], \quad \|\varphi_s(a,\overline{a},b,\overline b)\|_s\leq 2\|(a,\overline{a},b,\overline b)\|_s.
\end{equation}
The fact that $T(a,\overline{a},b,\overline b)\geq1$ implies that $\tau=\varphi_1$ is well defined. 
Thus, writing the equation~\eqref{phisprocheidentite} for $s=1$ and using the bound
\eqref{bootstrap}, we obtain
\begin{equation*}
  \|\tau(a,\overline{a},b,\overline b)-(a,\overline{a},b,\overline b)\|_s \lesssim \left\|(a,\overline{a},b,\overline b)\right\|_{s}^3.
\end{equation*}
Moreover  $\tau$ preserves the class $\Tc_{\rm res}^{s}(\mu)$ for any $\mu\leq 1$ as a consequence of the formula  $\nabla_z(f\circ\tau)(z)=(D\tau(z))^*\nabla f\circ\tau(z)$ and the fact that $D\tau(z)$ maps $Z_s$ into $Z_s$.
On the other hand $\{\chi_4,\Mbb\}=\{\chi_4,\Lbb\}=0$, therefore the flow $\tau=\varphi_1$ of $\chi_4$ conserves the mass and the momentum: 
$\Mbb(\tau(a,\bar a,b,\bar b))=\Mbb(a,\bar a,b,\bar b)$ and $\Lbb(\tau(a,\bar a,b,\bar b))=\Lbb(a,\bar a,b,\bar b)$.
We have thus proved the forth point of the proposition.

We recall that by construction, the remainder term $R_6$ is
\begin{equation*}
 R_6:=\left\{P_4,\chi_4\right\}+\displaystyle\int_0^1(1-t)\left\{\left\{H,\chi_4\right\},\chi_4\right\}\circ\varphi_t\mathrm{d}t.
\end{equation*}
The polynomials $P_4$, $\chi_4$ and $Z_4$ have bounded coefficients and have the prescribed form \eqref{P}. Therefore, using the homological equation~\eqref{eqhomol}, 
the same is true for the polynomials $Q_1$ and $Q_2$ defined by
\begin{equation*}
\begin{dcases}
 Q_1:=\left\{P_4,\chi_4\right\}, \\ 
 Q_2:=\left\{\left\{H,\chi_4\right\},\chi_4\right\}=\left\{Z_4,\chi_4\right\}-\left\{P_4,\chi_4\right\}+\left\{\left\{P_4,\chi_4\right\},\chi_4\right\}.
 \end{dcases}
\end{equation*}
Thus using Lemma \ref{XPanalytic}  we conclude that $R_6\in\Tc_{\rm res}^{s}(1)$.\\
Moreover, the polynomials $P_4$, $\chi_4$ and $Z_4$ are of order 4. Thus, the polynomials $Q_1$ and $Q_2$ are of order at least 6.
Therefore, the computations of Lemma \ref{XPanalytic} give
\begin{equation*}
\left\|X_{Q_1}(a,\overline{a},b,\overline b)\right\|_{s}\lesssim\left\|(a,\overline{a},b,\overline b)\right\|^5_{s} \; \text{and} 
\; \left\|X_{Q_2}(a,\overline{a},b,\overline b)\right\|_{s}\lesssim\left\|(a,\overline{a},b,\overline b)\right\|^5_{s}.
\end{equation*}
Finally, for $\varepsilon$ small enough and $(a,\overline{a},b,\overline b)\in B_{s}(\varepsilon)$, the estimate \eqref{bootstrap} implies
\begin{equation*}
\left\|X_{R_6}(a,\overline{a},b,\overline b)\right\|_{s}\lesssim\left\|(a,\overline{a},b,\overline b)\right\|^5_{s}.
\end{equation*}
On the other hand since $\tau$ conserves the mass and the momentum, $R_6=H\circ \tau-P_2-Z_4$ commutes with $\Lbb$ and $\Mbb$.
The proof of the point ($iii$) and thus the proof of the proposition is completed.
\end{proof}

\section{Two applications of KAM for the coupled NLS system}\label{secappli}

We aim to use  Theorem \ref{KAM} in order to study the extended coupled Schrödinger systems:
\begin{equation}\begin{dcases}
i\partial_tu+\partial_{xx}u=\left|v\right|^2u+\frac{\partial g}{\partial \overline{u}}(u,\overline{u},v,\overline{v}),& \quad (t,x)\in \mathbb{R}\times\mathbb{T}, \\ 
i\partial_tv+\partial_{xx}v=\left|u\right|^2v+\frac{\partial g}{\partial \overline{v}}(u,\overline{u},v,\overline{v}),&
\end{dcases}\label{sysgen}\end{equation}
where $g$  is a real\footnote{Here real means that $g(z_1,\overline{z}_1,z_2,\overline{z}_2)\in\R$.} analytic function on a neighborhood of the origin in $\C^4$ and $g$ is of order at least 5 in $(u,\overline{u},v,\overline{v})$. We set
$$R_5(a,\bar a,b,\bar b)=\int_\T g(\sum_\Z a_je^{ijx},\sum_\Z \bar a_je^{-ijx},\sum_\Z b_je^{ijx},\sum_\Z \bar b_je^{-ijx})dx$$
and we note that, in view of example \ref{exf} , $R_5\in\Tc^{s}({\mu})$ for some $\mu>0$. We will assume that 
\be\label{g}g(u,\overline{u},v,\overline{v})=f(|u|^2,|v|^2,u\overline v,v\overline u)\ee
in such a way  that $R_5$ commutes with $\Lbb$ and $\Mbb$. \\
With the notations introduced in \eqref{defNP}, the Hamiltonian of the system is thus given by
\begin{equation}\label{defH}
 H=P_2+P_4+R_5.
\end{equation}
After the application of the Birkhoff normal form of Proposition \ref{BNF}, we have
\begin{equation}\label{defHB}
 H^B=H\circ\tau=P_2+Z_4+R_5\circ\tau+R_6,
\end{equation}
where we recall that
\begin{equation*}
 Z_4(a,\overline{a},b,\overline b)=\sum\limits_{\left\{k,l\right\}=\left\{i,j\right\}}a_kb_l\overline{a}_i\overline b_j.
\end{equation*}
In order to apply the KAM result, the goal is to study the resonant term $Z_4$ to determine which part of $Z_4$ is an effective part, and which part can be treated as a remainder term.
Due to the assumptions of Theorem \ref{KAM}, the idea is to put the jetless part of $Z_4$ in the remainder term and to consider the jet part as the effective part of the Hamiltonian.

The idea is to consider the following solution of the linear system:
\begin{equation}\label{sollin}
 u(x,t)=a_pe^{ipx}e^{-ip^2t}, \quad v(x,t)=b_{q}e^{iqx}e^{-iq^2t}.
\end{equation}
We introduce the constants $\rho_i$, the variables $x_i$ and $\theta_i$,  the actions $I_i$ ($i=1,2$) and the variables $\zeta_\a,\ \a\in\Zc$ defined by
\begin{equation}\begin{dcases}\label{defIthetazeta}
a_{p}(t)=(\nu\rho_1+x_1(t))^\frac12e^{i\theta_1(t)}=:\sqrt{I_1}e^{i\theta_1(t)}, \\
b_{q}(t)=(\nu\rho_2+x_2(t))^\frac12e^{i\theta_2(t)}=:\sqrt{I_2}e^{i\theta_1(t)},\\
a_k(t)=\zeta_{k+}(t),\quad k\neq p,\\
b_k(t)=\zeta_{k-}(t), \quad k\neq q,
\end{dcases}
\end{equation}
where $(\rho_1,\rho_2)\in[1,2]^2$ and $\nu$ is a small parameter which controls the size of the solution. 
The canonical symplectic structure $-i(\mathrm{d}u\wedge\mathrm{d}\overline{u}+\mathrm{d}v\wedge\mathrm{d}\overline{v})$ becomes
\begin{equation*}
 -\mathrm{d}I\wedge\mathrm{d}\theta-i\mathrm{d}\zeta\wedge\mathrm{d}\overline{\zeta}= -\mathrm{d}x\wedge\mathrm{d}\theta-i\mathrm{d}\zeta\wedge\mathrm{d}\overline{\zeta}.
\end{equation*}
We note that in view of example \ref{exf} and Proposition \ref{BNF},  $R_5\circ\tau+R_6$ belongs to $\Tc^{s}(\s,{\mu})$ for some $\s>0$ and $\mu>0$. Further $R_5\circ\tau+R_6$ commutes with $\Lbb$ and $\Mbb$.

We want to study the linear stability of the torus $\mathbf{T}^{lin}_\rho$ defined by the solution~\eqref{sollin} of the linear system with $x=0$, $\theta$ real and $z=0$ 
(where we recall that $z=(\zeta,\bar\zeta)$).
This torus can be written as
\begin{equation*}
 \mathbf{T}^{lin}_\rho:=\left\{(I,\theta,z)\,|\, I=\nu\rho=(\nu\rho_1,\nu\rho_2),\theta\in\R^2/2\pi\Z^2,z=0\right\}.
\end{equation*}
For that purpose, we introduce a toroidal neighborhood of the torus $\mathbf{T}^{lin}_\r$ by
\begin{equation}\label{deftaurho}
 \mathbf{T}_\rho(\nu,\sigma,\mu,s):=\left\{(I,\theta,z)\,|\, |I-\nu\rho|<\nu\mu^2,|\Im\theta|<\sigma,\|z\|_s<\nu^\frac12\mu\right\}.
\end{equation}
Thanks to a translation in actions between $x$ and $I=\nu\rho+x$, 
we make an analogy between the toroidal neighborhood $\mathbf{T}_\rho(\nu,\sigma,\mu,s)$ of the torus $\mathbf{T}^{lin}_\r$ 
and the neighborhood $\mathcal{O}^s(\sigma,\nu^\frac12\mu)$ of the origin. Indeed we have
\begin{equation*}
 \mathbf{T}_\rho(\nu,\sigma,\mu,s)\approx\mathcal{O}^s(\sigma,\nu^\frac12\mu)=\left\{(x,\theta,z)\,|\, |x|<\nu\mu^2,|\Im\theta|<\sigma,\|z\|_s<\nu^\frac12\mu\right\}.
\end{equation*}

As a remark, we see that $\mathbf{T}_\rho(\nu,\sigma,\mu,s)$ is a small neighborhood of $\mathbf{T}^{lin}_\rho$ in the following sense:
\begin{equation*}
 \lim_{\sigma,\mu\rightarrow0}\left(\mathbf{T}_\rho(\nu,\sigma,\mu,s)\right)=\mathbf{T}^{lin}_\rho.
\end{equation*}

Through the two following subsections, we study two examples of application of the KAM theorem. For that purpose, we highlight in Subsection~\ref{subsecunstable} some 
unstable linear tori (corresponding to the case $p\neq q$), whereas we highlight in Subsection \ref{subsecstable} some stable ones (corresponding to the case $p=q$).

\subsection{1st case: emphasizing an unstable linear torus}\label{subsecunstable}

We consider the case $p\neq q$ and thus the following solution of the linear system:
\begin{equation}\label{sollin1}
 u(x,t)=a_pe^{ipx}e^{-ip^2t}, \quad v(x,t)=b_{q}e^{iqx}e^{-iq^2t}, \quad \text{with} \quad p\neq q.
\end{equation}
The goal of this section is to prove the following:

\begin{theorem}\label{thmuns} 
There exist $\nu_0>0$, $\sigma_0>0$ and $\mu_0>0$ such that, for $s>\frac12$, $0<\nu\leq\nu_0$, $0<\sigma\leq\sigma_0$, $0<\mu\leq\mu_0$  and $ \rho\in\D$

\begin{enumerate}[(i)]
 \item There exist 
 \begin{equation*}
  \Phi_\rho:\left(\begin{array}{lll}&\mathcal{O}^s(\frac12,\frac{e^{-\frac12}}{2})\rightarrow &\mathbf{T}_\rho(\nu,1,1,s)\\ 
 &(r,\theta,z)\mapsto &(I,\theta,z') \end{array}\right)
 \end{equation*}
real holomorphic transformations, analytically depending on $\rho$, which transform the symplectic structure 
$-\mathrm{d}r\wedge\mathrm{d}\theta-i\mathrm{d}\zeta\wedge\mathrm{d}\overline{\zeta}$ on $\mathcal{O}^s(\frac12,\frac{e^{-\frac12}}{2})$ to the symplectic structure
$-\nu\mathrm{d}I\wedge\mathrm{d}\theta-i\nu\mathrm{d}\zeta'\wedge\mathrm{d}\overline{\zeta'}$ on $\mathbf{T}_\rho(\nu,1,1,s)$.
The change of variables $\Phi_\rho$ is close to the scaling by the factor $\nu^\frac12$ on the $\L$-modes but not on the $\F$-modes, where it is close to a certain affine 
transformation depending on $\theta$.

 \item $\Phi_\rho$ puts the Hamiltonian $H=P_2+P_4+R_5$ in  normal form in the following sense:
\begin{equation}\label{H1}
\tilde H= \frac1\nu\left(H\circ\Phi_\rho-C\right)(r,\theta,z)=h_0(r,z)+f(r,\theta,z),
\end{equation}
where $C=\nu^2\rho_1\rho_2+\nu p^2\rho_1+\nu q^2\rho_2$ is a constant and the effective part $h_0$ of the Hamiltonian reads 
\begin{equation}\label{h01}
 h_0=\Omega(\rho)\cdot r+ \sum_{j\neq p,q}\left( \Lambda^a_j(\rho)|a_j|^2+\Lambda^b_j(\rho)|b_j|^2 \right)
 + \frac12\langle z_f,K(\rho)z_f \rangle.
\end{equation}
The frequencies $\Omega_\rho$ are given by
\begin{equation*}
 \Omega(\rho)=
\begin{pmatrix}
p^2+\nu\rho_2 \\ q^2+\nu\rho_1 
\end{pmatrix},
\end{equation*}
the eigenvalues $\Lambda^a_j$ and $\Lambda^b_j$ are defined by
\begin{equation*}
\Lambda^a_j(\rho)=j^2+\nu\rho_2 \quad \text{and} \quad \Lambda^b_j(\rho)=j^2+\nu\rho_1,
\end{equation*}
and the symmetric real matrix $K(\rho)$, acting on the four external modes 
$
 z_f:=\Phi_\rho\begin{pmatrix}
 b_{p}  \\ a_{q} \\ \overline b_p \\ \overline a_q
 \end{pmatrix},
$
is given by
\begin{equation*}
K(\rho)=
\nu\begin{pmatrix}
0  & \sqrt{\rho_1\rho_2} & \r_2-\r_1 & 0 \\ 
\sqrt{\rho_1\rho_2} & 0 & 0 & \r_1-\r_2  \\ 
\r_2-\r_1 & 0 &  0 & \sqrt{\rho_1\rho_2} \\ 
0 & \r_1-\r_2 &  \sqrt{\rho_1\rho_2} & 0 
\end{pmatrix}.
\end{equation*}

 \item The remainder term $f$ belongs to $\Tc_{\rm res}^{s}(\sigma,\mu,\D)$ and satisfies
 \begin{equation*}
  [f]^{s}_{\sigma,\mu,\D}\lesssim \nu \quad \text{and} \quad [f^T]^{s}_{\sigma,\mu,\D}\lesssim \nu^{\frac32}.
 \end{equation*}
\end{enumerate}
\end{theorem}

%
%

\subsubsection{Set up of the change of variables \texorpdfstring{$\Phi_\rho$}{phirho}}
The construction of the change of variables $\Phi_\rho$ is decomposed in two steps. First we eliminate the angles and then we rescale the variables.

\subsubsubsection{Structure of the Hamiltonian and elimination of the angles}

We study the Hamiltonian $H^B$, defined in equation \eqref{defHB}, obtained after the Birkhoff normal form iteration:
\begin{equation*}
 H^B=H\circ\tau=P_2+Z_4+R_5\circ\tau+R_6.
\end{equation*}
Here, the term $P_2$ already contributes to the desired effective Hamiltonian $h_0$ and the constant term $C$ of Theorem \ref{thmuns}, 
whereas the terms $R_5\circ\tau$ and $R_6$ contribute to the remainder term $f$. 
Therefore, we just have to deal with the resonant term $Z_4$, separating the constant, the effective and the remainder parts. 
We split the polynomial $Z_4$ according to the number of inner modes $a_p$ and $b_q$.
The term of order 4 in $a_p$, $b_q$ from $Z_4$ is given by 
\begin{align*}
 Z_{4,4}:=|a_p|^2|b_q|^2&=(\nu\rho_1+x_1)(\nu\rho_2+x_2)\\
 &=\underbrace{\nu^2\rho_1\rho_2}_{\text{constant}}+\underbrace{\nu\rho_2x_1+\nu\rho_1x_2}_{\text{effective part}}+\underbrace{x_1x_2}_{\text{remainder}}.
\end{align*}
Due to the structure of the resonant set, there is no term of order 3 in $a_p$, $b_q$ in $Z_4$.
The terms of order 2 in $a_p$, $b_q$ from $Z_4$ are 
\begin{align*}
 Z_{4,2}:=|a_p|^2\sum_{k\neq q}|b_k|^2+|b_q|^2\sum_{k\neq p}|a_k|^2+
 a_pb_q\overline a_{q}\overline b_p+\overline a_p\overline b_qa_qb_p.
 \end{align*}
For the first term (the same goes for the second one), we write 
\begin{equation*}
 |a_p|^2\sum_{k\neq q}|b_k|^2=\underbrace{\nu\rho_1\sum_{k\neq q}|b_k|^2}_{\text{effective part}}
 +\underbrace{x_1\sum_{k\neq q}|b_k|^2}_{\text{remainder}}.
\end{equation*}
The study of the two last terms is trickier because in each term there is four different modes and thus there are angles.
In order to split these terms between effective and remainder terms, we recall that we want in Theorem~\ref{thmuns} the remainder term $f$ to be small and to have a jet 
smaller. 
Therefore, the idea is to keep the jet part of these terms in the effective part and to put the other part (without jet) in the remainder part. We obtain for example for the 
term $a_pb_q\overline{a}_q\overline b_p$:
\begin{align*}
 \underbrace{\nu\sqrt{\rho_1\rho_2}e^{i(\theta_1+\theta_2)}\overline a_q\overline b_p}_{\text{effective part}}
 +\underbrace{\left(\sqrt{(\nu\rho_1+x_1)(\nu\rho_2+x_2)}-\nu\sqrt{\rho_1\rho_2}\right)e^{i(\theta_1+\theta_2)}\overline a_q\overline b_p}_{\text{remainder (jetless part)}}.
\end{align*}
The effective part of $Z_{4,2}$ is thus given by
\begin{align*}
 Z_{4,2}^e:=\nu\big(\rho_1\sum_{k\neq q}|b_k|^2+\rho_2\sum_{k\neq p}|a_k|^2
 +\sqrt{\rho_1\rho_2}(e^{i(\theta_1+\theta_2)}\overline a_q\overline b_p+e^{-i(\theta_1+\theta_2)}a_qb_p)\big).
\end{align*}
Then, in order to kill the angles, we introduce the symplectic change of variables 
\begin{equation*}
\Psi_{ang}\left(x,\theta,z=(a,b)\right)=\left(y,\theta,z'=(c,d)\right),  
\end{equation*}
where the new variables $c,d$ and $y$ are defined by
\begin{equation}
\begin{dcases}
  c_q=a_qe^{-i\theta_2}, \quad \quad c_k=a_k, \; k\neq p,q, \\
 d_p=b_pe^{-i\theta_1}, \quad \quad d_k=b_k, \; k\neq p,q, \\
 y_1=x_1+|b_p|^2, \quad \quad y_2=x_2+|a_q|^2.
\end{dcases} \label{defpsiang}
\end{equation}
\begin{remark}
Due to the symmetry of the term $e^{i(\theta_1+\theta_2)}\overline a_q\overline b_p+e^{-i(\theta_1+\theta_2)}a_qb_p$, we could have interchange $\theta_1$ and $\theta_2$ in the 
definition of the change of variables $\Psi_{ang}$ (and change the definition of $y_1$ and $y_2$ to keep the symplectism of $\Psi_{ang}$). 
Nevertheless, the choice made here has two main advantages:
\begin{enumerate}
 \item It allows to generalize the result for any finite number of (all different) modes excited for $u$ and $v$.  
 \item It makes the computations of the mass and momentum easier.
\end{enumerate}

\end{remark}

This change of variables $\Psi_{ang}$ is the reason why the change of variable $\Phi_\rho$ of Theorem \ref{thmuns} is not close to a scaling for the $\F$-modes $a_q$ and $b_p$.
Finally, the terms $Z_{4,1}$ and $Z_{4,0}$ of order 1 and 0 in $a_p$, $b_q$ from $Z_4$ are remainder terms.
We can thus write the Hamiltonian $\tilde{H}=H^B\circ\Psi_{ang}$ as
\begin{equation}\label{defhtilde}
 \tilde{H}=H^B\circ\Psi_{ang}=C+H^e+R,
\end{equation}
where the constant part $C$ is given by 
\begin{equation*}
C=\nu^2\rho_1\rho_2+\nu p^2\rho_1+\nu q^2\rho_2,
\end{equation*}
the remainder term $R$ is defined by
\begin{align*}
R=&Z_{4,0}\circ\Psi_{ang}+Z_{4,1}\circ\Psi_{ang}+(y_1-|d_p|^2)\sum_{k\neq q}|d_k|^2+(y_2-|c_q|^2)\sum_{k\neq p}|c_k|^2\\
&+\left(\sqrt{(\nu\rho_1+y_1-|d_p|^2)(\nu\rho_2+y_2-|c_q|^2)}-\nu\sqrt{\rho_1\rho_2}\right)(\overline{c}_q\overline{d}_{p}+c_qd_p)\\
&+(y_1-|d_p|^2)(y_2-|c_q|^2)+R_5\circ\tau\circ\Psi_{ang}+R_6\circ\Psi_{ang},
\end{align*}
and the effective Hamiltonian $H^e$ reads 
\begin{align*}
 H^e &=(p^2+\nu\rho_2)y_1+(q^2+\nu\rho_1)y_2+\nu(\r_1-\r_2)|d_p|^2+\nu(\r_2-\r_1)|c_q|^2\\
 &\quad + \nu\sqrt{\rho_1\rho_2}(\overline c_q\overline d_p+c_qd_p) + \sum_{k\neq p,q}(k^2+\nu\rho_2)|c_k|^2+(k^2+\nu\rho_1)|d_k|^2.
\end{align*}
The new frequencies are thus given by
\begin{equation*}
 \Omega(\rho)=\left( \begin{array}{c}
p^2+\nu\rho_2 \\
q^2+\nu\rho_1
\end{array} \right).
\end{equation*}
This expression of the Hamiltonian $\tilde{H}$ is really close to the desired one of Theorem~\ref{thmuns}. 
In order to control the size of the remainder term, the idea is now to rescale the variables $c$, $d$ and $y$.

\subsubsubsection{Rescaling of the variables and introduction of \texorpdfstring{$\Phi_\rho$}{phirho}}

In order to study the initial Hamiltonian $H$ on the torus $\mathbf{T}_\rho(\nu,\frac12,\frac{e^{-\frac12}}{2},s)$, we introduce the rescaling of the variables $y$ and $z$ 
by the change of variables 
\begin{equation*}
 \chi_\rho\left(y,\theta,z'\right)=(r,\theta,z),
\end{equation*}
where
\begin{equation*}
(r,\theta,z):=(\nu y,\theta,\nu^\frac12z').
\end{equation*}
The symplectic structure becomes 
\begin{equation*}
 -\nu\mathrm{d}r\wedge\mathrm{d}\theta-i\nu\mathrm{d}\zeta\wedge\mathrm{d}\overline{\zeta}.
\end{equation*}
By definition, the change of variables $\chi_\rho$ sends the neighborhood $\mathcal{O}^s(\sigma,\mu)$ of the origin into a toroidal neighborhood of $\mathbf{T}^{lin}_\rho$:
\begin{equation*}
\chi_\rho\left(\mathcal{O}^s(\sigma,\mu)\right)=\mathbf{T}_\rho(\nu,\sigma,\mu,s), \quad \forall \sigma,\mu,\nu,s>0.
\end{equation*}
We can now define the change of variable $\Phi_\rho$ of Theorem \ref{thmuns} by:
\begin{equation}\label{defphirho}
 \Phi_\rho:=\tau\circ\Psi_{ang}\circ\chi_\rho.
\end{equation}
Thanks to this definition, we are now able to prove Theorem \ref{thmuns}.

\subsubsubsection{Remainders are in the good class }

We use the clustering defined in Example \ref{ex22} and consider the class $\Tc^{s}(\s,{\mu})$ related to the new variables $(r,\theta,z)$ induced by $ \Phi_\rho$. The following lemma will insure that all the remainders are in the restricted class defined by \eqref{mom}.
\begin{lemma}\label{MoM} Let $f\in\Tc^{s}(\s,{\mu})$. If $f$ commutes with $\Mbb$ and $\Lbb$ then $f\in\Tc_{\rm res}^{s}(\s,{\mu})$.
\end{lemma}
\proof
We have
\begin{align*}\mathbb{M}&=\sum\limits_{j\in\mathbb{Z}}j(|a_j|^2+|b_j|^2)\\
&=\sum\limits_{j\neq p,q}j(|c_j|^2+|d_j|^2)+ p(\nu\rho_1+y_1)+q(\nu\rho_2+y_2)\\
&=\nu\sum\limits_{j\neq p,q}j(|\zeta_{j+}|^2+|\zeta_{j-}|^2)+ \nu p(\rho_1+r_1)+\nu q(\rho_2+r_2)
\end{align*}
and
 \begin{align*}\mathbb{L}&=\sum\limits_{j\in\mathbb{Z}}|a_j|^2+|b_j|^2\\
&=\nu\sum\limits_{j\neq p,q}(|\zeta_{j+}|^2+|\zeta_{j-}|^2)+ \nu(\rho_1+r_1+\rho_2+r_2).
\end{align*}
Thus if $\a=(j,\pm),\b=(\ell,\pm)\in\L$ with $[\a]=[\b]$  (i.e. $j=\pm\ell$) we have
\begin{align*}
\{e^{ik\cdot\theta} \zeta_\a\bar \zeta_\b,\Mbb\}&=-i\nu(pk_1+qk_2+j-\ell)e^{ik\cdot\theta} \zeta_a\bar \zeta_b
\end{align*}
and
\begin{align*}
\{e^{ik\cdot\theta} \zeta_\a\bar \zeta_\b,\Lbb\}&=-i\nu (k_1+k_2)e^{ik\cdot\theta} \zeta_a\bar \zeta_b.
\end{align*}
Thus $e^{ik\cdot\theta} \zeta_\a\bar \zeta_\b$ is in the jet of $f$ only if $pk_1+qk_2+j-\ell=0$ and $k_1+k_2=0$. In the case $j=\ell$ these conditions lead to $k=(0,0)$ since $p\neq q$. In the case $j=-\ell$ they imply $w_\a=w_\b=|j|\leq |(p,q)| |k|$.
\endproof

\subsubsection{Proof of Theorem \ref{thmuns}}

First, we can check the first point of the theorem by proving that
$\Phi_\rho\left(\mathcal{O}^s(\frac12,\frac{e^{-\frac12}}{2})\right)\subset \mathbf{T}_\rho(\nu,1,1,s)$.
By definition of the rescaling $\chi_\rho$, we have
\begin{equation*}
\chi_\rho\left(\mathcal{O}^s(\frac12,\frac{e^{-\frac12}}{2})\right)=\mathbf{T}_\rho(\nu,\frac12,\frac{e^{-\frac12}}{2},s).
\end{equation*}
Applying the change of variable $\Psi_{ang}$ (see \eqref{defpsiang}), we obtain
\begin{equation*}
 \Psi_{ang}\circ\chi_\rho\left(\mathcal{O}^s(\frac12,\frac{e^{-\frac12}}{2})\right)\subset \mathbf{T}_\rho(\nu,\frac12,\frac12,s).
\end{equation*}
Thus, applying the Birkhoff change of variables $\tau$ which is close to the identity (see Theorem \ref{BNF}), we obtain
\begin{equation*}
 \tau\circ\Psi_{ang}\circ\chi_\rho\left(\mathcal{O}^s(\frac12,\frac{e^{-\frac12}}{2})\right)=\Phi_\rho\left(\mathcal{O}^s(\frac12,\frac{e^{-\frac12}}{2})\right)
 \subset \mathbf{T}_\rho(\nu,1,1,s).
\end{equation*}
Then, we have to check the structure of the Hamiltonian $H\circ\Phi_\rho$. For that purpose, we start from the Hamiltonian $\tilde{H}=H\circ\tau\circ\Psi_{ang}$ defined in 
\eqref{defhtilde}, and we write
\begin{equation*}
H\circ\Phi_\rho=\tilde{H}\circ\chi_\rho=C+\nu h_0+\nu f,
\end{equation*}
where $h_0$ and $f$ are defined by
\begin{equation*}
 h_0:=\frac1\nu H^e\circ\chi_\rho \quad \text{and} \quad f:=\frac1\nu R\circ\chi_\rho.
\end{equation*}
By construction, $h_0$ satisfies the properties of the point (ii) of the theorem.

For the study of $f$, we first need to show that for $\sigma$ and $\mu$ small enough, we have $f\in\Tc_{\rm res}^{s}(\sigma,\mu,\D)$. 
We recall that $R$ is defined in the previous subsubsection. 
From the definition of $R$, we write 
\begin{equation}\label{decompo f}
 f=f_Z+f_{e}+f_5+f_6,
\end{equation}
where $f_Z$ is the part of $f$ that contains the terms $Z_{4,0}$ and $Z_{4,1}$, $f_{e}$ is the explicit part of $f$ and $f_5$ (respectively $f_6$) is the part with the term $R_5$ 
(respectively $R_6$). We remark that for all these terms, the explicit changes of variables $\Psi_{ang}$ and $\frac1\nu\chi_\r$ don't play a role here.
Applying Lemma~\ref{XPanalytic}, we first have $f_Z\in\Tc^{s}(1,1,\D)$. 
For the explicit part $f_{e}$, it is straightforward to check that $\nabla_z f(r,\theta,z,\rho)\in Z_{s}\times Z_s$ as soon as $z\in Z_{s}\times Z_s$. 
Therefore, we have $f_{e}\in\Tc^{s}(1,1,\D)$ too.
Now, by the third point of the Birkhoff normal form Proposition~\ref{BNF}, we also have $f_6\in\Tc^{s}(1,1,\D)$.
Using once again Proposition~\ref{BNF} for the behavior of the change of variable $\tau$ (forth point) and Example~\ref{exf}, there exists $\sigma_0>0$ and $\mu_0>0$
such that $f_5\in\Tc^{s}(\sigma,\mu,\D)$ for $0<\sigma\leq\sigma_0$ and $0<\mu\leq\mu_0$.\\
On the other hand by construction the four Hamiltonians $f_Z$, $f_{e}$, $f_5$ and $f_6$ commute with $\Lbb$ and $\Mbb$ therefore, by Lemma \ref{MoM}, they are all in the restricted class. Finally, we obtain 
\begin{equation*}
 f\in\Tc_{\rm res}^{s}(\sigma,\mu,\D) \quad \text{for} \quad 0<\sigma\leq\sigma_0 \quad \text{and} \quad 0<\mu\leq\mu_0.
\end{equation*}

We fix now $0<\sigma\leq\sigma_0$ and $0<\mu\leq\mu_0$.

Then, for the estimates on the norms of $f$, we remark that $R\circ\chi_\rho$ contains only terms of order at least 2 in $\nu$ 
(for example, $Z_{4,0}\circ\Psi_{ang}\circ\chi_\rho$ and all the other terms, except those with $R_5$ or $R_6$ which are smaller, are of order $\nu^2$), thus we do have 
\begin{equation*}
  [f]^{s}_{\sigma,\mu,\D}\lesssim \nu.
 \end{equation*}
 
Finally, for the estimate on the jet part of $f$, we first remark that by construction
\footnote{Here we use that the Hamiltonian $G=(\sqrt{(\r_1+y_1-|c_q|^2)(\r_2+y_2-|d_p|^2)}-\sqrt{\r_1\r_2})c_qd_p$ satisfies $G^T=0$.},
\begin{equation*}
 f^T=\frac1\nu (R_5\circ\tau\circ\Psi_{ang}\circ\chi_\rho)^T+\frac1\nu(R_6\circ\Psi_{ang}\circ\chi_\rho)^T.
\end{equation*}
Therefore, using that $R_5$ and $R_6$ are of order 5 and 6, we have
\begin{equation*}
 [f^T]^{s}_{\sigma,\mu,\D}\lesssim \nu^{\frac32}.
\end{equation*}
 
\qed

\subsubsection{Study of the $\F$-modes}

Assume that the hypothesis A0, A1 and A2 are satisfied (see Appendix \ref{AppA}), we can apply Theorem~\ref{KAM}.
To study the linear stability of the torus $\Phi(\mathbf{T}^{lin})$, with $\Phi$ defined in  Theorem \ref{KAM}, we thus have to check whether there exists hyperbolic directions or not.
In the effective Hamiltonian $h_0$ defined in Theorem \ref{thmuns}, the modes $a_j$ and $b_j$ ($j\neq p,q$) are elliptic modes 
and thus don't have influence on this linear stability. 
Therefore, in order to study the linear stability of the torus $\Phi(\mathbf{T}^{lin})$, we have to study the semi-external modes $(d_p,c_q)$ to determine 
if they are elliptic or hyperbolic modes.
The variables $(d_p,c_q,\overline{d}_p,\overline c_q)$ satisfy
\begin{equation*}
\begin{pmatrix}
\dot{d}_p \\
\dot{c}_q \\
\dot{\overline{d}}_p \\
\dot{\overline{c}}_q
\end{pmatrix}
=-i\begin{pmatrix}
\frac{\partial h_0}{\partial \overline{d}_p} \\
\frac{\partial h_0}{\partial \overline{c}_q} \\
-\frac{\partial h_0}{\partial d_p} \\
-\frac{\partial h_0}{\partial c_q}
\end{pmatrix}
=M
\begin{pmatrix}
d_p \\
c_q \\
\overline{d}_p \\
\overline{c}_q
\end{pmatrix},
\end{equation*}
where we denote by $M=-iJK(\rho)$ the matrix 
\begin{equation*}
M:=-i\nu\begin{pmatrix}
\r_1-\r_2 & 0 & 0 & \sqrt{\rho_1\rho_2}\\
0 & \r_2-\r_1 & \sqrt{\rho_1\rho_2} & 0 \\
0 & -\sqrt{\rho_1\rho_2} & \r_2-\r_1 & 0\\
-\sqrt{\rho_1\rho_2} & 0 & 0 & \r_1-\r_2\\
\end{pmatrix}.
\end{equation*}
A straightforward computation shows that the characteristic polynomial $\chi_M$ is  
\begin{equation*}
 \chi_M(\lambda)=\left((\lambda-i(\r_2-\r_1))^2-\rho_1\rho_2\right)
 \left((\lambda+i(\r_2-\r_1))^2-\rho_1\rho_2\right).
\end{equation*}
Thus, the eigenvalues of M are 
\begin{align}\label{l12}
 \lambda_{1,2} = i(\r_2-\r_1)\pm \sqrt{\rho_1\rho_2}, \quad \quad \lambda_{3,4} = -i(\r_2-\r_1)\pm \sqrt{\rho_1\rho_2}.
\end{align}
Therefore, we have four eigenvalues $\lambda_i$, $i=1..4$ which satisfy
\begin{equation*}
 \Re(\lambda_i)\neq0, \quad \text{for} \quad i=1\ldots4.
\end{equation*}
For $\r_1\neq \r_2$, we have four different eigenvalues and we can diagonalize the matrix $M$.
For $\r_1= \r_2$, we have two double eigenvalues: $\pm \nu\sqrt{\rho_1\rho_2}$ but the matrix $M$ is still diagonalizable (we can check by example that the two vectors 
$^t(1,0,0,i)$ and $^t(0,1,-i,0)$ are eigenvectors associated to the eigenvalue $\nu\sqrt{\rho_1\rho_2}$, another way to prove this is to use the fact that all real skew matrix is 
diagonalizable with purely imaginary eigenvalues).
Therefore, in both cases, we have two hyperbolic directions, 
this implies the linear instability of the torus $\mathcal{T}:=\left\{y=0\right\}\times\mathbb{T}^2\times\left\{z=0\right\}$.

\begin{remark} The matrix $M$ does not depend on the choice of the excited modes $p\neq q$.\end{remark}

\subsubsection{Structure of \texorpdfstring{$h_0$}{hzero}}
In order to apply the KAM theorem, let us see that we can write the Hamiltonian $h_0$ with the normal structure \eqref{NN}.
By equation \eqref{h01}, we have 
\begin{equation*}
 h_0=\Omega(\rho)\cdot r+ \sum_{j\neq p,q}\left( \Lambda^a_j(\rho)|a_j|^2+\Lambda^b_j(\rho)|b_j|^2 \right)
 + \frac12\langle z_f,K(\rho)z_f \rangle.
\end{equation*}
We use the notations introduced in Example \ref{ex22}. First we remark for the elliptic part that 
\begin{equation*}
 \sum_{j\neq p,q}\left( \Lambda^a_j(\rho)|a_j|^2+\Lambda^b_j(\rho)|b_j|^2 \right)=
 \sum_{\alpha\in\L}\Lambda_\a(\rho)|\zeta_\a|^2,
\end{equation*}
where 
\begin{equation*}
 \Lambda_{j\pm}(\rho)=j^2+\nu\rho_\pm \text{ with } \rho_+=\rho_2 \text{ and } \rho_-=\rho_1.
\end{equation*}
For the part related to the matrix $K(\rho)$, the study of the matrix $M$ in the previous subsection and the eigenvectors associated to the eigenvalues computed in \eqref{l12} suggest 
the introduction of the following symplectic change of variables
\begin{equation}\label{def_ef}
\begin{dcases}
 \zeta_e=\frac1{\sqrt2}(c_q+i\bar d_p), \qquad  \bar \zeta_e=\frac1{\sqrt2}(\bar c_q+id_p), \\
 \zeta_f=\frac1{\sqrt2}(d_p+i\bar c_q), \qquad  \bar \zeta_f=\frac1{\sqrt2}(\bar d_p+ic_q),
\end{dcases}
\end{equation}
where the variables $(d_p,c_q)$ are defined in \eqref{defpsiang}. 
We remark here that $\bar \zeta_e$ and $\bar \zeta_f$ are not the complex conjugates of $\zeta_e$ and $\zeta_f$, but the Hamiltonian dual variables of $\zeta_e$ and $\zeta_f$ 
in the following sense (for example for $\zeta_e$):
\begin{equation*}
 \partial_t \zeta_e = -i\frac{\partial h_0}{\bar \zeta_e}, \qquad \partial_t \bar\zeta_e = i\frac{\partial h_0}{\zeta_e}.
\end{equation*}
In the new variables, we have
\begin{align*}
 \frac12\langle z_f,K(\rho)z_f \rangle :&= (\r_2-\r_1)(|d_p|^2-|c_q|^2) + \nu\sqrt{\rho_1\rho_2}(\overline c_q\overline d_p+c_qd_p) \\
 &= (\r_1-\r_2-i\nu\sqrt{\rho_1\rho_2})|\zeta_e|^2+(\r_2-\r_1-i\nu\sqrt{\rho_1\rho_2})|\zeta_f|^2.
\end{align*}
Finally, we can write 
\begin{equation}\label{h01nf}
 h_0=\Omega(\rho)\cdot r+ \sum_{\alpha\in\Zc=\L\cup\F}\Lambda_\a(\rho)|\zeta_\a|^2,
\end{equation}
where $\Omega(\rho)= \begin{pmatrix}
p^2+\nu\rho_2 \\ q^2+\nu\rho_1 
\end{pmatrix}$, $\zeta_\a$ is defined by \eqref{defIthetazeta} for $\a\in\L=\Z\setminus\{p,q\}\times\{\pm\}$, 
$\zeta_\a$ is defined by \eqref{def_ef} for $\a\in\F=\{e,f\}$, and 
\begin{align*}
\begin{dcases}
\Lambda_\a(\rho)&= j^2+\nu\rho_{\pm} \text{ for } \a=(j,\pm)\in\L \text{ with } \rho_{+}=\rho_2 \text{ and } \rho_{-}=\rho_1, \\
\Lambda_\a(\rho)&=
\begin{dcases}
\r_1-\r_2 -i\nu\sqrt{\rho_1\rho_2} \text{ for } \a=e\in\F, \\
\r_2-\r_1 -i\nu\sqrt{\rho_1\rho_2} \text{ for } \a=f\in\F.
\end{dcases}
\end{dcases}
\end{align*}

\subsubsection{Application of our KAM result.} \label{subsubAppli}
To apply Theorem \ref{KAM} to the Hamiltonian $\tilde H$ given by \eqref{H1} it remains to verify Hypothesis A0, A1 and A2 for $h_0$ given by \eqref{h01nf}. 
This is done in Appendix \ref{AppA} where we prove that Hypothesis A1 and A2 are satisfied for $\delta=\frac12\nu$. Therefore, we obtain
\begin{theorem}\label{KAMunstable}
Fix $p\neq q$. There exists $\nu_0>0$ and for $0<\nu<\nu_0$ there exists $\Cc_\nu\subset [1,2]^2$ asymptotically of full measure 
(i.e. $\lim_{\nu\to 0}\meas([1,2]^2\setminus \Cc_\nu )=0$) such that for $\r\in\Cc_\nu$ the torus 
$\Tc_{\nu\r}:=\{|a_p|^2=\nu \r_1,\ |b_q|^2=\nu \r_2, \text{all other modes vanishing}  \}$, which is invariant for the Hamiltonian flow associated to $P_2$, 
persists in slightly deformed way under the perturbation $P_4+R_5$. Furthermore this invariant torus is linearly unstable.
\end{theorem}
We can formulate our result in terms of small amplitude quasi periodic solutions. 
We notice that, in view of Theorem \ref{thmuns} we have 
\begin{equation*}
 \eps= [f^T]^{s}_{\frac12,\frac{e^{-\frac12}}{2},\D}=\O(\nu^{3/2}),
\end{equation*}
 therefore Theorem \ref{coro1} is a corollary of Theorem \ref{KAMunstable}.

%

\subsection{2nd case: emphasizing a stable linear torus}\label{subsecstable}

Here, we study the missing case of the previous subsection: the case $p=q$. Therefore, we start with the following solution of the linear system:
\begin{equation}\label{sollin2}
 u(x,t)=a_{p}e^{ipx}e^{-ip^2t}, \quad v(x,t)=b_{p}e^{ipx}e^{-ip^2t}.
\end{equation}
In this  case to insure that the Hamiltonian functions are in the restricted class defined by \eqref{mom}, we have to replace assumption \eqref{g} by 
\begin{equation}\label{g'}
g(u,\bar u,v,\bar v)=f(|u|^2,|v|^2)
\end{equation}
in such a way that $R_5$ commutes with the partial mass
\begin{align}
\label{L1} \Lbb_u&=\sum_{j\in\Z} |a_j|^2\\
\label{L2} \Lbb_v&=\sum_{j\in\Z} |b_j|^2.
\end{align}
The goal of this section is to prove the following:

\begin{theorem}\label{thmsta} 
There exist $\nu_0>0$, $\sigma_0>0$ and $\mu_0>0$ such that, for $s>\frac12$, $0<\nu\leq\nu_0$, $0<\sigma\leq\sigma_0$, $0<\mu\leq\mu_0$  and $ \rho\in\D$
\begin{enumerate}[(i)]
 \item There exist
 \begin{equation*}
  \Phi_\rho:\left(\begin{array}{lll}&\mathcal{O}^s(\frac12,\frac{e^{-\frac12}}{2})\rightarrow &\mathbf{T}_\rho(\nu,1,1,s)\\ 
 &(r,\theta,z)\mapsto &(I,\theta,z') \end{array}\right)
 \end{equation*}
real holomorphic transformations, analytically depending on $\rho$, which transform the symplectic structure 
$-\mathrm{d}r\wedge\mathrm{d}\theta-i\mathrm{d}\zeta\wedge\mathrm{d}\overline{\zeta}$ on $\mathcal{O}^s(\frac12,\frac{e^{-\frac12}}{2})$ to the symplectic structure
$-\nu\mathrm{d}I\wedge\mathrm{d}\theta-i\nu\mathrm{d}\zeta'\wedge\mathrm{d}\overline{\zeta'}$ on $\mathbf{T}_\rho(\nu,1,1,s)$.
The change of variables $\Phi_\rho$ is close to a certain affine transformation depending on $\theta$.

 \item $\Phi_\rho$ puts the Hamiltonian $H=P_2+P_4+R_5$ in  normal form in the following sense:
\begin{equation}\label{H2}
 \frac1\nu\left(H\circ\Phi_\rho-C\right)(r,\theta,z)=h_0(r,z)+f(r,\theta,z),
\end{equation}
where $C=\nu^2\rho_1\rho_2+\nu p^2(\rho_1+\rho_2)$ is a constant and the effective part $h_0$ of the Hamiltonian reads 
\begin{equation}\label{h02}
 h_0=\Omega(\rho)\cdot r + \sum_{j\neq p}\left( \Lambda^a_j(\rho)|a_j|^2+\Lambda^b_j(\rho)|b_j|^2 \right).
\end{equation}
The frequencies $\Omega_\rho$ are given by
\begin{equation*}
 \Omega(\rho)=
\begin{pmatrix}
p^2+\nu\rho_2 \\ p^2+\nu\rho_1 
\end{pmatrix},
\end{equation*}
the eigenvalues $\Lambda^a_j$ and $\Lambda^b_j$ are defined by
\begin{equation*}
\Lambda^a_j(\rho)=j^2-p^2+\nu\sqrt{\rho_1\rho_2} \quad \text{and} \quad \Lambda^b_j(\rho)=j^2-p^2-\nu\sqrt{\rho_1\rho_2}.
\end{equation*}

 \item The remainder term $f$ belongs to $\Tc_{\rm res}^{s}(\sigma,\mu,\D)$ and satisfies
 \begin{equation*}
  [f]^{s}_{\sigma,\mu,\D}\lesssim \nu \quad \text{and} \quad [f^T]^{s}_{\sigma,\mu,\D}\lesssim \nu^{\frac32}.
 \end{equation*}
\end{enumerate}
\end{theorem}

\subsubsection{Set up of the change of variables \texorpdfstring{$\Phi_\rho$}{phirho}}
As in the previous subsection, we have to eliminate the angles, and then to perform a rescaling of the variables.

\subsubsubsection{Structure of the Hamiltonian and elimination of the angles}
As the linear term $P_2$ contributes already to the constant part $C$ and to the effective Hamiltonian $h_0$, the main difference with the study of two different modes is the behavior 
of the resonant term $Z_4$. We split the polynomial $Z_4$ according to the number of inner modes $a_p$ and $b_p$.
The term of order 4 in $a_{p}$, $b_{p}$ from $Z_4$ is given by 
\begin{equation*}
 Z_{4,4}:=|a_{p}|^2|b_{p}|^2=(\rho_1+x_1)(\rho_2+x_2)=\underbrace{\nu^2\rho_1\rho_2}_{\text{constant}}+\underbrace{\nu\rho_2x_1+\nu\rho_1x_2}_{\text{effective part}}
+\underbrace{x_1x_2}_{\text{remainder}}.
\end{equation*}
The terms of order 2 in $a_{p}$, $b_{p}$ from $Z_4$ are 
\begin{align*}
 Z_{4,2}:=|a_{p}|^2\sum_{k\neq p}|b_k|^2+|b_{p}|^2\sum_{k\neq p}|a_k|^2+
 a_{p}\overline b_{p}\sum_{k\neq p}\overline a_{k}b_{k}+\overline a_pb_{p}\sum_{k\neq p}a_{k}\overline b_{k}.
\end{align*}
Separating the effective part (with jet) from the remainder part (without jet), we show that the effective part of $Z_{4,2}$ is given by
\begin{align*}
 Z_{4,2}^e:=&\nu\rho_1\sum_{k\neq p}|b_k|^2+\nu\rho_2\sum_{k\neq p}|a_k|^2\\
 &+\nu\sqrt{\rho_1\rho_2}\left(e^{i(\theta_1+\theta_2)}\sum_{k\neq p}\overline a_{k}b_{k}+e^{-i(\theta_1+\theta_2)}\sum_{k\neq p}\overline a_pb_{p}\right).
\end{align*}
Then, in order to kill the angles, we introduce the symplectic change of variables 
\begin{equation*}
\Psi_{ang}\left(x,\theta,z=(a,b)\right)=\left(y,\theta,z'=(c,d)\right),  
\end{equation*}
where the new variables $c,d$ and $y$ are defined by
\begin{equation}
\begin{dcases}
 c_{k}&=a_{k}e^{-i\theta_1}, \quad \quad d_{k}=b_{k}e^{-i\theta_2}, \quad k\neq p, \\
 y_1&=x_1+\sum_{k\neq p}|a_{k}|^2, \quad \quad y_2=x_2+\sum_{k\neq p}|b_{k}|^2.
\end{dcases} \label{defpsiang2}
\end{equation}
Finally, the term $Z_{4,0}$ of order 0 in $a_p$, $b_p$ from $Z_4$ is still a remainder term.
We can thus write the Hamiltonian $\tilde{H}=H^B\circ\Psi_{ang}$ as
\begin{equation}\label{defhtilde2}
 \tilde{H}=H^B\circ\Psi_{ang}=C+H^e+R,
\end{equation}
where the constant part $C$ is given by 
\begin{equation*}
C=\nu^2\rho_1\rho_2+\nu p^2(\rho_1+\rho_2),
\end{equation*}
the remainder term $R$ is defined by
\begin{align*}
R=&Z_{4,0}\circ\Psi_{ang}+(y_1-\sum_{k\neq p}|c_{k}|^2)\sum_{k\neq q}|d_k|^2+(y_2-\sum_{k\neq p}|d_{k}|^2)\sum_{k\neq p}|c_k|^2\\
&+\alpha(\nu,\rho,z')\sum_{k\neq p}(\overline c_{k}d_{k}+\overline d_{k}c_{k})+(y_1-\sum_{k\neq p}|c_{k}|^2)(y_2-\sum_{k\neq p}|d_{k}|^2)\\
&+R_5\circ\tau\circ\Psi_{ang}+R_6\circ\Psi_{ang},
\end{align*}
with
\begin{equation*}
 \alpha(\nu,\rho,z')=\sqrt{(\nu\rho_1+y_1-\sum_{k\neq p}|c_{k}|^2)(\nu\rho_2+y_2-\sum_{k\neq p}|d_{k}|^2)}-\nu\sqrt{\rho_1\rho_2},
\end{equation*}
and the effective Hamiltonian $H^e$ reads 
\begin{align*}
 H^e =&(p^2+\nu\rho_2)y_1+(p^2+\nu\rho_1)y_2
 +\sum_{k\neq p}(k^2-p^2)(|c_{k}|^2+|d_{k}|^2)\\
 &+\nu\sqrt{\rho_1\rho_2}\left(\sum_{k\neq p}\overline{c_{k}}d_{k}+\sum_{k\neq p}c_{k}\overline{d_{k}}\right).
\end{align*}
The new frequencies are thus given by
\begin{equation*}
 \Omega(\rho)=\left( \begin{array}{c}
p^2+\nu\rho_2 \\
p^2+\nu\rho_1
\end{array} \right).
\end{equation*}
The last term in $H^e$ is not a diagonal term (\textit{i.e} in $|c_{k}|^2$ and $|d_{k}|^2$).
Nevertheless, due to its symmetries, the good idea is to introduce a new symplectic change of variables
\begin{equation*}
\Psi_{sym}\left(I',\theta,z'=(c,d)\right)=\left(I',\theta,z''=(e,f)\right),  
\end{equation*}
where
\begin{equation*}
 e_k=\frac{c_k+d_k}{\sqrt2}, \quad f_k=\frac{c_k-d_k}{\sqrt2}.
\end{equation*}
A good way to see how this change of variables appears is to look at the equations satisfied by $(c_k,\bar c_k,d_k,\bar d_k)$ (for $k\neq p$). We have
\begin{equation*}
 i\begin{pmatrix}
\dot{c}_k \\
\dot{\overline{c}}_k \\
\dot{d}_k \\
\dot{\overline{d}}_k
\end{pmatrix}
=\begin{pmatrix}
\frac{\partial H^e}{\partial \overline{c}_k} \\
-\frac{\partial H^e}{\partial c_k} \\
\frac{\partial H^e}{\partial \overline{d}_k} \\
-\frac{\partial H^e}{\partial d_k}
\end{pmatrix}
=\begin{pmatrix}
k^2-p^2 & 0 & \sqrt{\rho_1\rho_2}  & 0\\
0 & p^2-k^2 & 0 & -\sqrt{\rho_1\rho_2} \\
\sqrt{\rho_1\rho_2} & 0 & k^2-p^2 & 0\\
0 & -\sqrt{\rho_1\rho_2} & 0 & p^2-k^2\\
\end{pmatrix}
\begin{pmatrix}
c_k \\
\overline{c}_k \\
d_k \\
\overline{d}_k
\end{pmatrix}.
\end{equation*} 
The diagonalization of the matrix $M$, defined by
\begin{equation*}
M:=-i\begin{pmatrix}
k^2-p^2 & 0 & \sqrt{\rho_1\rho_2}  & 0\\
0 & p^2-k^2 & 0 & -\sqrt{\rho_1\rho_2} \\
\sqrt{\rho_1\rho_2} & 0 & k^2-p^2 & 0\\
0 & -\sqrt{\rho_1\rho_2} & 0 & p^2-k^2\\
\end{pmatrix},
\end{equation*}
allows to introduce the variables $e_k=\frac{c_k+d_k}{\sqrt2}$ and $f_k=\frac{c_k-d_k}{\sqrt2}$.
In these new variables, we obtain 
\begin{align*}
 H^e\circ\Psi_{sym} =&(p^2+\nu\rho_2)y_1+(p^2+\nu\rho_1)y_2
 +\sum_{k\neq p}(k^2-p^2+\nu\sqrt{\rho_1\rho_2})|e_{k}|^2 \\
  &+\sum_{k\neq p}(k^2-p^2-\nu\sqrt{\rho_1\rho_2})|f_{k}|^2.
\end{align*}
To finish the construction of the change of variables $\Phi_\rho$, we just have now to rescale the variables $e$, $f$ and $y$.

\subsubsubsection{Rescaling of the variables and introduction of \texorpdfstring{$\Phi_\rho$}{phirho}}

As in the previous case, we introduce the rescaling by the change of variables
\begin{equation*}
 \chi_\rho\left(r,\theta,z''\right)=(x,\theta,z),
\end{equation*}
where
\begin{equation*}
(x,\theta,z):=(\nu r,\theta,\nu^\frac12z'').
\end{equation*}
The symplectic structure becomes 
\begin{equation*}
 -\nu\mathrm{d}r\wedge\mathrm{d}\theta-i\nu\mathrm{d}\zeta\wedge\mathrm{d}\overline{\zeta}.
\end{equation*}
By definition of $\chi_\rho$ we have
\begin{equation*}
\chi_\rho\left(\mathcal{O}^s(\sigma,\mu)\right)=\mathbf{T}_\rho(\nu,\sigma,\mu,s).
\end{equation*}
We can now define the change of variable $\Phi_\rho$ of Theorem \ref{thmsta} by:
\begin{equation}\label{defphirho2}
 \Phi_\rho:=\tau\circ\Psi_{ang}\circ\Psi_{sym}\circ\chi_\rho.
\end{equation}

\subsubsubsection{Remainders are in the good class }

We use the clustering defined in Example \ref{ex21} and consider the class $\Tc^{s}(\s,{\mu})$ related to the new variables $(r,\theta,z)$ induced by $ \Phi_\rho$. The following lemma will insure that all the remainders are in the restricted class defined by \eqref{mom}.
\begin{lemma}\label{MoM'} Let $f\in\Tc^{s}(\s,{\mu})$. 
If $f$ commutes with the partial mass $\Lbb_u$ and $\Lbb_v$ then $f\in\Tc_{\rm res}^{s}(\s,{\mu})$.
\end{lemma}
\proof
We have
 \begin{align*}\Lbb_u&=\sum_{j\in\mathbb{Z}}|a_j|^2=\nu(\rho_1+r_1)\\
\Lbb_v&=\sum_{j\in\mathbb{Z}}|b_j|^2=\nu(\rho_2+r_2).
\end{align*}
Thus if $\a=(j,\pm),\b=(\ell,\pm)\in\L$ with $[\a]=[\b]$  (i.e. $j=\pm\ell$) we have
\begin{align*}
\{e^{ik\cdot\theta} \zeta_\a\bar \zeta_\b,\Lbb_u\}&=-i\nu k_1e^{ik\cdot\theta} \zeta_a\bar \zeta_b\\
\{e^{ik\cdot\theta} \zeta_\a\bar \zeta_\b,\Lbb_v\}&=-i\nu k_2e^{ik\cdot\theta} \zeta_a\bar \zeta_b.
\end{align*}
Thus $e^{ik\cdot\theta} \zeta_\a\bar \zeta_\b$ is in the jet of $f$ only if $k=(0,0)$. \endproof

\subsubsection{Proof of Theorem \ref{thmsta}}

The only difference with the proof of Theorem \ref{thmuns} lies in the construction of the change of variables $\Phi_\rho$.
First, we can remark that the change of variables $\Psi_{sym}$ doesn't change the norm. Indeed, writing 
$\Psi_{sym}\left(I',\theta,z'=(c,d)\right)=\left(I',\theta,z''=(e,f)\right)$, we remark that 
\begin{equation*}
 |c_{k}|^2+|d_{k}|^2=|e_{k}|^2+|f_{k}|^2, \quad \forall k\neq p.
\end{equation*}
Thus, we have 
\begin{equation*}
\Psi_{sym}\circ\chi_\rho\left(\mathcal{O}^s(\frac12,\frac{e^{-\frac12}}{2})\right)=\mathbf{T}_\rho(\nu,\frac12,\frac{e^{-\frac12}}{2},s).
\end{equation*}
The change of variables $\Psi_{ang}$ (see \eqref{defpsiang2}) is constructed as in \eqref{defpsiang}, and the change of variables $\tau$ from the Birkhoff normal form is the same 
as in the proof of Theorem \ref{thmuns}. Therefore, we have
\begin{equation*}
 \tau\circ\Psi_{ang}\circ\Psi_{sym}\circ\chi_\rho\left(\mathcal{O}^s(\frac12,\frac{e^{-\frac12}}{2})\right)=\Phi_\rho\left(\mathcal{O}^s(\frac12,\frac{e^{-\frac12}}{2})\right)
 \subset \mathbf{T}_\rho(\nu,1,1,s).
\end{equation*}
Therefore, to conclude the proof of Theorem \ref{thmsta}, it suffices to write
\begin{equation*}
H\circ\Phi_\rho=\tilde{H}\circ\Psi_{sym}\circ\chi_\rho=C+\nu h_0+\nu f,
\end{equation*}
with $h_0$ and $f$ defined by
\begin{equation*}
 h_0:=\frac1\nu H^e\circ\Psi_{sym}\circ\chi_\rho \quad \text{and} \quad f:=\frac1\nu R\circ\Psi_{sym}\circ\chi_\rho.
\end{equation*}
Thus defined, $h_0$ satisfies the point (ii) of the theorem. For the third point of the theorem about the term $f$, the same study as in previous case (thanks to a decomposition as 
in equation \eqref{decompo f}) shows that there exists $\sigma_0>0$ and $\mu_0>0$ such that 
\begin{equation*}
 f\in\Tc^{s}(\sigma,\mu,\D) \quad \text{for} \quad 0<\sigma\leq\sigma_0 \quad \text{and} \quad 0<\mu\leq\mu_0,
\end{equation*}
and for $0<\sigma\leq\sigma_0$ and $0<\mu\leq\mu_0$, we also have 
\begin{equation*}
  [f]^{s}_{\sigma,\mu,\D}\lesssim \nu \quad \text{and} \quad [f^T]^{s}_{\sigma,\mu,\D}\lesssim \nu^{\frac32}.
 \end{equation*} 
 On the other hand, $g$ has the form \eqref{g'}, we easily verify that $f$ commutes with $\Lbb_1$ and $\Lbb_2$ and thus $f$ is in the restricted class $\Tc^{s}_{\rm res}(\sigma,\mu,\D)$ by Lemma \ref{MoM'}.\qed


\subsubsection{Structure of \texorpdfstring{$h_0$}{hzero}}
As $\F=\emptyset$, it's easy to write the Hamiltonian $h_0$ with the normal structure \eqref{NN}.
By equation \eqref{h02}, we have 
\begin{equation}\label{h02nf}
 h_0=\Omega(\rho)\cdot r+ \sum_{\alpha\in\Zc=\L}\Lambda_\a(\rho)|\zeta_\a|^2,
\end{equation}
where $\Omega(\rho)= \begin{pmatrix}
p^2+\nu\rho_2 \\ p^2+\nu\rho_1 
\end{pmatrix}$, $\zeta_\a$ is defined by \eqref{defIthetazeta} (we recall that we have here $p=q$) for $\a\in\L=\Z\setminus\{p\}\times\{\pm\}$, and 
\begin{equation*}
\Lambda_\a(\rho)= j^2-p^2\pm\nu\sqrt{\rho_1\rho_2} \quad \text{for} \quad \a=(j,\pm)\in\L.
\end{equation*}

\subsubsection{Application of our KAM result.}
We see in Appendix \ref{AppA} that the hypotheses A0, A1 and A2 are satisfied for the Hamiltonian \eqref{H2} for the parameter 
$\delta=\frac12\nu$. 
Thus we can apply Theorem \ref{KAM} and we obtain
\begin{theorem}\label{KAMstable}
Fix $p$. There exists $\nu_0>0$ and for $0<\nu<\nu_0$ there exists $\Cc_\nu\subset [1,2]^2$ asymptotically of full measure 
(i.e. $\lim_{\nu\to 0}\meas([1,2]^2\setminus \Cc_\nu )=0$) 
such that for $\r\in\Cc_\nu$ the torus $\Tc_{\nu\r}:=\{|a_p|^2=\nu \r_1,\ |b_p|^2=\nu \r_2, \text{all other modes vanishing}  \}$, 
which is invariant for the Hamiltonian flow associated to $P_2$, persists in slightly deformed way under the perturbation $P_4$. 
Furthermore this invariant torus is linearly stable.

\end{theorem}

In terms of small amplitude quasi periodic solutions, Theorem \ref{KAMstable} reads

\begin{corollary}\label{coro2}Fix $p$ and integer and $s>1/2$. There exists $\nu_0>0$ and for $0<\nu<\nu_0$ there exists $\Cc_\nu\subset [1,2]^2$ asymptotically of full measure (i.e. $\lim_{\nu\to 0}\meas([1,2]^2\setminus \Cc_\nu )=0$) such that for $\r\in\Cc_\nu$ there exists a quasi periodic solution $(u,v)$ of \eqref{sysgen} of the form
\begin{equation*}\begin{dcases}
 u(x,t)&=\sum_{j\in\mathbb{Z}}u_j(t\om)e^{ijx},\\
 v(x,t)&=\sum_{j\in\mathbb{Z}}v_j(t\om)e^{ijx},
 \end{dcases}\end{equation*}
where $U(\cdot)=(u_j(\cdot))_{j\in\Z}$ and $V(\cdot)=(v_j(\cdot))_{j\in\Z}$ are analytic  functions from $\T^2$ into $\ell^2_s$ satisfying uniformly in $\theta\in\T^2$
\begin{equation*}
\begin{dcases}
\big||u_p(\theta)|-\sqrt{\nu\rho_1}\big|^2+\sum_{j\neq p}(1+j^2)^s|u_j(\theta)|^2=\mathcal{O}(\nu^3) ,\\
\big||v_p(\theta)|-\sqrt{\nu\rho_2}\big|^2+\sum_{j\neq p}(1+j^2)^s|v_j(\theta)|^2=\mathcal{O}(\nu^3) 
\end{dcases}
\end{equation*}
and where $\om\equiv\om(\r)\in\R^2$ is a nonresonant frequency vector that satisfies
\begin{equation*}
 \om=(p^2,p^2)+\mathcal{O}(\nu^\frac32).
\end{equation*}
Furthermore this solution is linearly stable.
\end{corollary}

\appendix 
\section{Verification of the non resonance hypotheses}\label{AppA}
In this appendix, the goal is to check that the hypotheses A0, A1 and A2 are satisfied for the Hamiltonians $h_0$ given by  \eqref{h01nf} and \eqref{h02nf}. The first one, Hypothesis A0, is trivially satisfied and we focus on A1 and A2.

In this process we will use the conservation of the mass and of the momentum (see Remark \ref{momentum}). 
The expression of $\Lbb$ and $\Mbb$ depends on the change of variable $\Phi_\r$ so we have to distinguish between the two examples studied in section 4.
We recall that we have initially 
\begin{equation*}
 \mathbb{L}=\sum\limits_{j\in\mathbb{Z}}(|a_j|^2+|b_j|^2), \quad \text{and} \quad 
 \mathbb{M}=\sum\limits_{j\in\mathbb{Z}}j(|a_j|^2+|b_j|^2)\,.
\end{equation*}

\noindent{\it The first case corresponding to \eqref{h01nf}.}\\
In the new variables, by Lemma \ref{MoM}, we have for the mass and momentum:
\begin{align*}
 \mathbb{L}&=\nu \sum\limits_{\a\in\L}|\zeta_\a|^2+\nu( \r_1+\r_2+r_1+r_2)
\end{align*}
and 
\begin{align*}
 \mathbb{M}=\nu\big(\sum_{\a\in \L}\a_1|\zeta_\a|^2+ p(r_1+\rho_1)+q(r_2+\rho_2)\big).
\end{align*}
So the perturbation $f$ of \eqref{H1} commutes with
\begin{equation}\label{com1}
\Lbb_1=r_1+r_2 + \sum\limits_{\a\in\L}|\zeta_\a|^2\quad \text{ and } \quad \Mbb_1=  pr_1+qr_2+\sum_{\a\in \L}\a_1|\zeta_\a|^2.
\end{equation}

\noindent{\it The second case corresponding to \eqref{h02nf}.}\\
In the new variables, by Lemma \ref{MoM'}, we have for the partial mass
\begin{align*}
 \mathbb{L}_u=\nu( \r_1+r_1) \quad \text{and} \quad  \mathbb{L}_v=\nu( \r_2+r_2).
\end{align*}
Therefore, the perturbation $f$ of \eqref{H2}, that commutes with these partial mass, doesn't have terms with angles. 
Indeed, we have for example 
\begin{equation*}
 \left\{e^{ik\cdot \theta}\bar\zeta_{\a}\zeta_\b,\mathbb{L}_u\right\}=i\nu k_1e^{ik\cdot \theta}\bar\zeta_{\a}\zeta_\b=0
 \quad \Rightarrow \quad k_1=0.
\end{equation*}
Thus, we just have to check Hypothesis A1 for this case.

\subsection{Verification of Hypothesis A1}

Conditions (a) and (b) trivially hold true with $\delta\leq\nu<1/5$ for $h_0$ given by  \eqref{h01nf} or by \eqref{h02nf}. 
For \eqref{h01nf} condition (c) also holds trivially true  with $\delta\leq1/2$ as soon as $\nu<1/4$. Nevertheless condition (c) is not always true for \eqref{h02nf}, for instance we have 
$$\Lambda_{-p,+}+\Lambda_{-p,-}=\nu\sqrt{\r_1\r_2}-\nu\sqrt{\r_1\r_2}=0.$$ 
More generally $\Lambda_{j,+}+\Lambda_{k,-}=j^2+k^2-2p^2$ can vanish for some values of $(j,k,p)$ and these are the only problems in order to verify the condition (c). The small divisor $\Lambda_{j,+}+\Lambda_{k,-}$ corresponds to the quadratic term $\zeta_{j,+}\zeta_{k,-}$. As we know that the perturbation $f$ commutes with $\Mbb_2$ we have to consider only the case $\{\Mbb_2,\zeta_{j,+}\zeta_{k,-}\}=0$. This yields the two conditions on $(j,k,p)$
$$j^2+k^2-2p^2=0\quad \text{ and }\quad j+k-2p=0$$
whose only solution is $k=j=p$ which is not possible since $(p,\pm)\notin\Zc$.

\subsection{Verification of Hypotheses A2 (i), (ii) and (iii)}

For these three hypotheses, the second alternative allows to obtain restrictions for the choice of $k$.
Thanks to equation \eqref{h01nf},  we have for $k\in\Z^2\setminus0$ and $\mathfrak z=|k|^{-1}(k_2,k_1)$
$$\nazz (\Omega(\r)\cdot k)=\nu|k|.$$
Therefore, since $|\nazz \Lambda_\a(\r)|\leq \nu$ for all $\a\in\Zc$, we get choosing $\delta\leq\nu$:
\begin{itemize}
\item the second part of alternative A2 (i) is always satisfied for $k\neq 0$,
\item the second part of alternative A2 (ii) is always satisfied  for $|k|\geq 2$,
\item the second part of alternative A2 (iii) is always satisfied for $|k|\geq 3$.
\end{itemize}
So it remains to verify  alternative A2 (ii)-(iii) for a finite number of choices of $k$:   $0<|k|< 2$ for (ii) and  $0<|k|< 3$ for (iii). 
The condition $0<|k|< 2$ implies that 
\begin{equation}\label{k<2}
 k\in \mathcal{K}_2:=\left\{\pm(0,1),\pm(1,0),(\pm1,\pm1)\right\},
\end{equation}
and the condition $0<|k|< 3$ implies that
\begin{equation}\label{k<3}
 k\in \mathcal{K}_3:=\mathcal{K}_2\cup\left\{\pm(0,2),\pm(2,0),(\pm1,\pm2),(\pm2,\pm1),(\pm2,\pm2)\right\}.
\end{equation}
As in the previous section we will use the conservation of the mass and momenta. The strategy is to to show that for each small divisor, either the small divisor corresponds to a 
term that doesn't exists (we use the conservation of the mass and momenta to show this) and we don't need to estimate it, either the small divisor satisfies one of the two alternatives 
of the hypotheses A2 (ii) and (iii).

\medskip

The small divisor $\Omega\cdot k+\Lambda_\a$ (with $\a=(j,\pm)\in\L$) corresponds to the monomial $e^{ik\cdot \theta}\bar\zeta_{\a}$ with $k\in\mathcal{K}_2$ defined in \eqref{k<2}, in the following sense:
\begin{equation*}
 \left\{e^{ik\cdot \theta}\bar\zeta_{\a},h_0\right\}=i(\Omega\cdot k+\Lambda_\a)e^{ik\cdot \theta}\bar\zeta_{\a}.
\end{equation*}
The conservation of the mass $\Lbb_1$ defined in \eqref{com1} gives 
\begin{equation*}
 \left\{e^{ik\cdot \theta}\bar\zeta_{\a},\Lbb_1\right\}=ie^{ik\cdot \theta}\bar\zeta_{\a}(k_1+k_2+1)=0.
\end{equation*}
Therefore, we just have to study the cases $k\in\left\{(0,-1),(-1,0)\right\}$. 
Moreover, the conservation of the momentum $\Mbb_1$ defined in \eqref{com1} gives 
\begin{equation*}
 \left\{e^{ik\cdot \theta}\bar\zeta_{\a},\Mbb_1\right\}=ie^{ik\cdot \theta}\bar\zeta_{\a}(pk_1+qk_2+j)=0.
\end{equation*}
For $k\in\left\{(0,-1),(-1,0)\right\}$, the conservation of the momentum $\Mbb_1$ implies $j\in\left\{p,q\right\}$, which is 
excluded here because $\a=(j,\pm)\in\L$.

We consider the small divisor $\Omega\cdot k+\Lambda_\a+ \Lambda_\b$ in the same way and we notice that it corresponds to the monomial 
$e^{ik\cdot \theta}\bar\zeta_{\a}\bar\zeta_{\b}$ with $k\in\mathcal{K}_3$ defined in~\eqref{k<3}.
Let $\a=(j,\pm_1), \b=(\ell,\pm_2) \in \L$.
We have 
\begin{equation}
 \Om(\r)\cdot k+\Lambda_\a(\r)+ \Lambda_\b(\r)=N_1(p,q,j,\ell)+\nu(k_1\r_2+k_2\r_1+\eta_1(\r)),
\end{equation}
where 
\begin{equation*}
 \eta_1(\r)\in\left\{2\r_1,2\r_2,\r_1+\r_2\right\} \quad \text{and} \quad N_1(p,q,j,\ell)=p^2k_1+q^2k_2+j^2+\ell^2.
\end{equation*}
The conservation of the mass $\Lbb_1$ implies $k_1+k_2=-2$. 
Thus, we just have to consider the three cases $k\in\left\{(0,-2),(-2,0),(-1,-1)\right\}$.
In these three cases, we have $|k_1\r_2+k_2\r_1+\eta_1(\r)|\in\left\{0,|\r_1-\r_2|,2|\r_1-\r_2|\right\}\in[0,2]$. Therefore, in order to check Hypothesis A2 (ii), it suffices to choose 
$\nu$ small enough ($\nu\leq\frac14$ is enough here) and to show that the integer part $N_1$ of the small divisor $\Om(\r)\cdot k+\Lambda_\a(\r)+\Lambda_\b(\r)$ is an nonzero integer.
For that purpose, we use the conservation of the momentum which gives
\begin{equation}\label{N1consmoment}
 pk_1+qk_2+j+\ell=0.
\end{equation}
Let us start with the case $k=(-1,-1)$. Using the conservation of the momentum \eqref{N1consmoment}, the integer part $N_1$ of the small divisor is zero if and only if
\begin{equation*}
 p+q=j+\ell \quad \text{and} \quad p^2+q^2=j^2+\ell^2,
\end{equation*}
which leads to $(p-j)(p+j)=(p-j)(l+q)$. Therefore, either we have $j=p$, which is excluded because $\a=(j,\pm_1)\in \L=\Z\setminus\{p,q\}\times\{\pm\}$; either we have 
$p+j=l+q$, which, using the conservation of the momentum \eqref{N1consmoment} once again, leads to $j=q$ which is also excluded for the same reason.
Thus, we have $|N_1(p,q,j,\ell)|\geq1$ (it is an nonzero integer). 
For the case $k=(-2,0)$, always using the conservation of the momentum, we have $N_1=0$ if and only if
\begin{equation*}
 2p=j+\ell \quad \text{and} \quad 2p^2=j^2+\ell^2,
\end{equation*}
which leads to $j=\ell=p$, and is excluded as we work with $\a,\b \in \L$.
The same computations for the case $k=(0,-2)$ lead to $j=\ell=q$. 
We thus obtain the same contradiction. Finally, with $\nu$ small enough, we obtain in the three cases
\begin{align*}
 |\Om(\r)\cdot k+\Lambda_\a(\r)+\Lambda_\b(\r)|\geq\frac12\geq\nu.
\end{align*}

The last small divisor to study here is $\Omega\cdot k+\Lambda_\a-\Lambda_\b$, corresponding to the monomial $e^{ik\cdot \theta}\bar\zeta_{\a}\zeta_{\b}$, with $k\in\mathcal{K}_3$.
Let $\a=(j,\pm_1), \b=(\ell,\pm_2) \in \L$.
The small divisor to study is given by
\begin{equation}
 \Om(\r)\cdot k+\Lambda_\a(\r)-\Lambda_\b(\r)=N_2(p,q,j,\ell)+\nu(k_1\r_2+k_2\r_1+\eta_2(\r)),
\end{equation}
where 
\begin{equation*}
 \eta_2(\r)\in\left\{0,\r_1-\r_2,\r_2-\r_1\right\} \quad \text{and} \quad N_2(p,q,j,\ell)=p^2k_1+q^2k_2+j^2-\ell^2.
\end{equation*}
This time, the conservation of the mass $\Lbb_1$ gives the equation $k_1+k_2=0$. 
Therefore we have to study in $\mathcal K_3$ the four cases $k\in\left\{\pm(-1,1),\pm(-2,2)\right\}$.
For $\a=(j,\pm_1)\in\L$ and $\b=(\ell,\pm_2)\in\L$, we split the study with respect to the values of $k$.
On the one hand, if $k\in\left\{(-1,1),(-1,1)\right\}$, we use the same strategy as in the study of the previous divisor: we show that $N_2$ is an nonzero integer.
This time, the conservation of the momentum which gives
\begin{equation*}
 pk_1+qk_2+j-\ell=0.
\end{equation*}
Using this equation for $k=\pm(1,-1)$, the integer $N_2$ is zero if and only if $p+q=j+l$. 
This condition leads, thanks to the conservation of the momentum once again, to $\{j,\ell\}=\{p,q\}$, which is excluded for $\a,\b \in \L$.
On the other hand, if $k\in\left\{(-2,2),(-2,2)\right\}$, the same strategy doesn't fit. Indeed, the couple $(j,\ell)=\frac12(3p-q,3q-p)$ gives $N_2=0$ and respects the conservation 
of the momentum. Therefore, we have to consider the second alternative of the hypothesis. For $\mathfrak z=|k|^{-1}(k_2,k_1)$, we have 
\begin{equation*}
 \nazz(\Om(\r)\cdot k+\Lambda_\a(\r)-\Lambda_\b(\r))\geq\nu(|k|-\frac{|k_1-k_2|}{|k|})=\sqrt2\nu.
\end{equation*}

\subsection{Verification of Hypothesis A2 (iv)}

This hypothesis occurs only for the Hamiltonian $h_0$ given by \eqref{h01nf} (there is no hyperbolic direction in the Hamiltonian given by \eqref{h02nf}). 
This time, there is no second alternative. Therefore, we can't reduce the choices of $k$. 
Nevertheless, 
using that $\Lbb_1$ and $\Mbb_1$ defined in \eqref{com1} do not depend on the hyperbolic modes, the terms
$e^{ik\cdot \theta}\bar\zeta_{\a}\zeta_\b$ and $e^{ik\cdot \theta}\zeta_{\a}\zeta_\b$ (with $\a,\b\in\F$) can appear if and 
only if 
\begin{equation*}
 k_1+k_2=0 \quad \text{and} \quad pk_1+qk_2=0.
\end{equation*}
Once again, using that $p\neq q$, this leads to $k=(0,0)$, and we don't have terms to control for this hypothesis.
\medskip

{\it Conclusion.} 
In order to show that the spectrum of $h_0$ satisfies the hypotheses A1 and A2, it suffices to choose 
for both cases ($h_0$ defined by \eqref{h01nf} or  \eqref{h02nf}) the value 
\begin{equation*}
\delta=\frac12\nu.
\end{equation*}
Therefore, the hypotheses of Theorem \ref{KAM} are satisfied, and we can apply this theorem to obtain Theorem \ref{thmuns} for the linear 
instability of the torus 
\begin{equation*}
 \Tc^u_{\nu\r}:=\{|a_p|^2=\nu \r_1,\ |b_q|^2=\nu \r_2, \text{all other modes vanishing}  \} \text{ where } p\neq q,
\end{equation*}
and to obtain Theorem \ref{thmsta} for the linear stability of 
\begin{equation*}
 \Tc^s_{\nu\r}:=\{|a_p|^2=\nu \r_1,\ |b_p|^2=\nu \r_2, \text{all other modes vanishing}  \}
\end{equation*}
for the extended cubic coupled Schr\"odinger systems \eqref{sysgen}.

\end{document}